\documentclass[draft, 11pt]{amsart}
\usepackage{amsmath, amsthm, amscd, amssymb, amsfonts, amsxtra, verbatim}
\usepackage{a4}

\usepackage{color}

\def\zp{M_{p,a}^{b,c}(\mathfrak{a})}

\def\pbgs{{\varepsilon }}

\def\MSpin{\operatorname{MSpin}}
\def\RMSpin{\widetilde{\operatorname{M}}\operatorname{Spin}}
\def\qedbox{\hbox{$\rlap{$\sqcap$}\sqcup$}}

\newcommand{\N}{\mathbb N}
\newcommand{\Z}{\mathbb Z}
\newcommand{\Q}{\mathbb Q}
\newcommand{\R}{\mathbb R}
\newcommand{\C}{\mathbb C}

\newcommand{\G}{\Gamma}
\newcommand{\g}{\gamma}
\newcommand{\ld}{\lambda}
\newcommand{\Ld}{\Lambda}
\newcommand{\vep}{\varepsilon}

\newcommand{\f}{\frac}
\newcommand{\tf}{\tfrac}

\newcommand{\arr}{\rightarrow}

\newcommand{\sk}{\smallskip}
\newcommand{\msk}{\medskip}

\newcommand{\I}{\text{\sl Id}}

\newcommand{\man}{M_\Gamma}
\newcommand{\GL}{\text{GL}}
\newcommand{\on}{\text{O}(n)}
\newcommand{\son}{\text{SO}(n)}
\newcommand{\spin}{\text{Spin}(n)}
\newcommand{\s}{\text{S}}

\numberwithin{equation}{section}

\theoremstyle{plain}
\newtheorem{theorem}{Theorem}[section]
\newtheorem{proposition}[theorem]{Proposition}
\newtheorem{lemma}[theorem]{Lemma}
\newtheorem{corollary}[theorem]{Corollary}
\theoremstyle{definition}
\newtheorem{definition}[theorem]{Definition}
\newtheorem{remark}[theorem]{Remark}
\newtheorem{example}[theorem]{Example}

\begin{document}

\makeatother
\def\Op{\operatorname{Op}}

\title[eta invariants and equivariant bordism]
{The eta invariant and equivariant bordism of flat manifolds with cyclic holonomy group of odd prime order}
\begin{author}[P.B. Gilkey, R.J. Miatello, R.A. Podest\'a]{Peter B.
Gilkey$^{(*)}$, Roberto J. Miatello$^{(**)}$ and Ricardo A.
Podest\'a$^{(**)}$}\end{author}
\begin{address}{PG: Mathematics Department, University of Oregon, Eugene, OR
97403, USA}\end{address}
\begin{email}{gilkey@uoregon.edu}
\end{email}
\begin{address}{RM, RP: FaMAF-CIEM, Universidad Nacional de C\'ordoba, 5000
C\'ordoba, Argentina}
\end{address}
\begin{email}{miatello@mate.uncor.edu, podesta@mate.uncor.edu}\end{email}

\begin{abstract}
We study the 
eta invariants of compact flat spin manifolds of dimension $n$ with holonomy group $\mathbb{Z}_p$, where $p$ is an odd prime. We find explicit expressions for the twisted and relative eta invariants and show that the reduced eta invariant is always an integer, except in a single case, when $p=n=3$. We use the expressions obtained to show that any such manifold is trivial in the appropriate reduced equivariant spin bordism group.
\end{abstract}

\keywords{Flat manifolds, eta invariant, equivariant bordism.
\newline $^{(*)}$ Partially supported by Project MTM2006-01432 (Spain)
\newline $^{(**)}$ Partially supported by CONICET, Foncyt  and SECyT-UNC}

\maketitle

\section{Introduction}\label{intro}
If $M$ is  a Riemannian manifold having finite holonomy group $F$, we shall say that $M$ is an \emph{$F$-manifold}.
As it is well known, any such  manifold is flat, by the Ambrose-Singer theorem \cite{AS}. Let $p$ be an odd prime. Throughout this paper, $M$ will be a compact flat Riemannian $n$-manifold with cyclic holonomy group of order $p$, $F \simeq \Z_p$, that is, in the above terminology, a flat $\Z_p$-manifold. Such a manifold is of the form $M_\G =\G \backslash \R^n$, with $\G$ a Bieberbach group such that $\Ld \backslash \G \simeq \Z_p$, where $\Ld$ denotes the translation lattice of $\G$. Flat $\Z_p$-manifolds have been fully classified by Charlap \cite{Ch65}, who used Reiner's classification \cite{Re} of integral representations of the group $\Z_p$. A convenient description of these manifolds will be given in \S \ref{S2-ZPM} and \S \ref{zpabc}.

\sk
It turns out that any $\Z_p$-manifold $M=M_\Gamma$ is spin, that is, it admits spin structures defined on its tangent bundle. Actually, we show that such $M$ admits exactly $2^{\beta_1}$ spin structures, where $\beta_1 = \beta_1(M)$ is the first Betti number of $M$.
In particular, we shall see that there is a unique spin structure such that the corresponding homomorphism
$\varepsilon : \Gamma \rightarrow \textrm{Spin} (n)$ is trivial on the translation lattice $\Lambda$ of $\Gamma$,
called the spin structure of \textit{trivial type}, or \textit{trivial} spin structure, for short.

\msk
One of the main goals of this paper is to obtain a rather explicit expression for the reduced eta invariant of an arbitrary spin $\Z_p$-manifold $M$, associated to the spinorial Dirac operator twisted by characters. We will make use of results in \cite{MiPo06}, where the spectra of twisted Dirac operators on flat bundles over arbitrary  compact flat spin manifolds is determined. Also, we refer to \cite{MR09} for a survey on the spectral geometry of flat manifolds and for a more complete bibliography than we can present here.

\sk
In a general setting, let $M$ be an arbitrary Riemannian $n$-manifold, let $D$ be a self adjoint partial differential operator of Dirac type, and
we let $Spec_D(M)=\{\lambda_n\}$ be the set of eigenvalues of $D$, counted with multiplicities.
Results of Seeley \cite{Se-67} show that the \textit{eta series}
\begin{equation}\label{etaseries}
\eta_D(s) := \sum_{\lambda_i\ne
0}\operatorname{sign}(\lambda_i)|\lambda_i|^{-s}
\end{equation}
is holomorphic for $\operatorname{Re}(s)>n$.
Furthermore, $\eta_D(s)$ has a meromorphic extension to $\mathbb{C}$ called the \textit{eta function} (that we still denote by $\eta_D(s)$) with isolated simple poles on the real axis. In their study of the index theorem for manifolds with boundary, Atiyah, Patodi, and Singer showed that $0$ is a regular value of the eta function (\cite{APS76-I,APS76-II,APS76-III}, see also \cite{Gi81})
and defined the \textit{eta invariant}
\begin{equation*}\label{eta invariant APS}
\eta_D := \eta_D(0),
\end{equation*}
as a measure of the spectral asymmetry of $D$, and the invariant 
\begin{equation}\label{etainvariant}
\bar \eta_D := \tf12 \big( \eta_D + \dim \ker D \big) \,,
\end{equation}
which is also referred to as the eta invariant by some authors.

\msk
We now return to $\Z_p$-manifolds. Let $\vep_h$ be a spin structure on a  $\Z_p$-manifold $M$. We consider the \textit{spin Dirac operator} $D_{\ell}$ twisted by a character $\rho_\ell$ of $\Z_p$, for $0\le \ell \le p-1$ (with $\ell=0$ corresponding to the untwisted case), acting on smooth sections of the twisted spinor bundle of $(M,\vep_h)$ (see \S\ref{S2-STDO}). Denote the associated eta series by $\eta_{\ell,h}(s)$ and the corresponding eta invariants by $\eta_{\ell,h}$ and $\bar \eta_{\ell,h}$, respectively.

\sk
In \S \ref{specasymm} we study the
spectrum $Spec_{D_{\ell}}(M,\vep_h)$ and determine its contribution to formula \eqref{etaseries}. We show that non trivial eta series can only occur for the so called \textit{exceptional} $\Z_p$-manifolds, in the terminology of Charlap \cite{Ch88}. Using the information on the spectrum in \S \ref{specasymm}, in Theorem \ref{thm.eta series} we obtain explicit expressions for the eta function $\eta_{\ell,h}(s)$ in terms of Hurwitz zeta functions $\zeta(s,\f jp)$, with $1\le j \le p-1$ (see \eqref{hurwitz}).
From these expressions for $\eta_{\ell,h}(s)$ we get formulas for the  eta invariants $\eta_{\ell,h}$, by evaluation at $s=0$ (Theorem \ref{thm.etainvs}).

It is known that the dimension of the kernel of $D_{\ell}$ on compact flat manifolds is non zero only for the spin structure of trivial type \cite{MiPo06}. In Proposition \ref{harmonics} we give an expression for $\dim \ker D_{\ell}$, which, together with our previous computations, and in light of  \eqref{etainvariant}, yield an expression for
$\bar \eta_{\ell,h}$ and for the difference 
$\bar\eta_{\ell,h} - \bar\eta_{0,h}$ 
of $M$ (see Remark \ref{remark} and Corollary \ref{eta untwisted}).

\sk
The integrality of $\bar \eta_{\ell,h}$, except in a single case, 
is one of the main results in this paper.


\begin{theorem}\label{twisted etas mod Z}
Let $p$ be an odd prime and let $\ell \in \N_0$, with $0\le \ell \le p-1$.
Consider a $\Z_p$-manifold $M$ of dimension $n$ equipped with a spin structure $\vep$.
Then
\begin{equation*}
\bar \eta_{\ell} \equiv 0 \mod \Z
\end{equation*}
unless $p=n=3$, and in this case $\bar \eta_{\ell} \equiv \f 23 \mod \Z$.
Furthermore, in all cases,
$\bar \eta_{\ell} - \bar \eta_{0}  \equiv 0  \mod \Z$.
\end{theorem}

\goodbreak
The theorem says that the eta invariants $\bar \eta_{\ell}$, $0\le \ell \le p-1$, are integers except in the case of the so called \textsl{tricosm}, the only 3-dimensional $\Z_3$-manifold (see Example \ref{tricosm}).

\sk
We point out that for certain $\Z_p$-manifolds with $p\equiv 3$ mod 4, there is a close connection between the spectral invariants and invariants coming from number theory. More precisely, when $\beta_1(M) = 1$ and $(n-1)/(p-1)$ is odd, the  eta invariants $\eta_{\ell}$ are given in terms of sums involving Legendre symbols $(\f jp)$ with $0\le j \le p-1$ (see Remark \ref{rem. hp}). Moreover, in the untwisted case, $\eta_{0}$ is a simple multiple of the class number $h_{-p}$ of the imaginary quadratic field $\Q(\sqrt {-p})$. Namely, these manifolds have only 2 spin structures and we have
$$\eta_{0,1} =  -4 \, p^{\f{a-1}2} \, \tfrac{h_{-p}}{\omega_{-p}}, \qquad \eta_{0,2}  = \big( \big( \tf{2}p \big) -1 \big) \, \eta_{0,1},$$
where $\omega_{-p}$ is the number of $p^\mathrm{{th}}$-roots of unity in $\Q(\sqrt {-p})$.

\msk
As a main application, in Section \ref{sect-5} we use the computations in Section \ref{sect-4} to study the equivariant spin bordism group of  $\Z_p$-manifolds. We recall that two compact closed spin manifolds $M_1$ and $M_2$ are said to be \textit{bordant} if there exists a compact spin manifold $N$ so that the boundary  of $N$ is  $M_1 \cup -M_2$ with the inherited spin structure, where $-M_2$ denotes $M_2$ with the opposite orientation. If  $M_1$ and $M_2$ have equivariant $\Z_p$-structures (see Section \ref{sect-5}), we say $N$ is an \textit{equivariant bordism} between $M_1$ and $M_2$ if the $\Z_p$-structure extends over $N$. In this situation $M_1$ and $M_2$ are said to be \textit{equivariant Spin-bordant}. Bordism is an important topological concept first investigated by Thom. Theorem \ref{thm-1.2} below states that any $\Z_p$-manifold with the canonical $\Z_p$-structure is equivariant bordant to the same manifold with the trivial $\Z_p$-structure; i.e., vanishes in the reduced equivariant bordism group. We postpone until Section 5 a more precise description.

As it is known, the eta invariant is an analytic spectral invariant, that gives rise to topological invariants which
completely detect the equivariant $\mathbb{Z}_p$ spin bordism groups. The integrality results of Theorem \ref{twisted etas mod Z} then yield the following geometric and topological result, one of the main motivations for this investigation.

\begin{theorem}\label{thm-1.2}
Let $(M, \vep,\sigma_p)$ and $(M, \vep,\sigma_0)$ denote a $\Z_p$-manifold $M$ which is equipped with
a spin structure $\vep$ together with the canonical and the trivial equivariant
$\mathbb{Z}_p$-structures
$\sigma_p$ and $\sigma_0$ respectively. Then $$[(M,  \vep,\sigma_p)] - [(M, \vep,\sigma_0)]=0$$ in
the reduced equivariant spin bordism group $\RMSpin_n(B\mathbb{Z}_p)$.
\end{theorem}


It is worth putting these groups into a bit of a historical context. The equivariant spin bordism groups are
important in algebraic topology as they are closely related to Brown-Peterson homology. In \cite{BBDG} the eta invariant was used to compute $BP_*(BG)$ where $G$ was a spherical space form group; this computation yields the additive structure of $\RMSpin_*(B\mathbb{Z}_p)$ for $p$ an odd prime. But in addition to their topological importance, these groups have also appeared in a geometric setting, for instance, in connection with spin manifolds with finite fundamental group admitting a metric of positive scalar curvature (see Remark \ref{peters-remark}).

\goodbreak

\msk
A brief outline of the paper is as follows. In Section 2 we start by giving a somewhat detailed description of the structure of $\Z_p$-manifolds and of their spin structures. Sections 3 through 5 are devoted to the proofs of the main results. In Sections 3 we study the spectrum and the eta series, in Section 4 we give the results concerning eta invariants and in Section 5 we settle the result on spin bordism. In these proofs we use a number of auxiliary formulas, stated and proved in Section 6. Namely, we need formulas for trigonometric products (\S \ref{sect-a-STP}), for twisted character Gauss sums (\S \ref{Sect-A-TCGS}) and sums involving Legendre symbols (\S \ref{Sect-a-SILS}). We have presented this material at the end to avoid interrupting the flow of our discussion of the main results.

\section{$\Z_p$-manifolds} \label{sect-2}
\subsection{Compact flat manifolds}\label{S2-CFM}
Any compact flat $n$-manifold is isometric to a quotient of the form
$$M_\G = \G\backslash\R^n$$
where $\G$ is a Bieberbach group, that is, a discrete, cocompact, torsion-free subgroup $\G$ of $\text{I}(\R^n)$, the isometry group of $\R^n$. Thus, one has that any element $\g \in \text{I}(\R^n) \simeq \on \rtimes \R^n$ decomposes uniquely as $\g = B L_b$, where $B \in \on$, $L_b$ denotes translation by $b\in \R^n$, and furthermore, multiplication is given by
\begin{equation}\label{semidirectproduct}
BL_b \cdot CL_c = BCL_{C^{-1}b+c} \,.
\end{equation}

The pure translations in $\G$ form a normal, maximal abelian subgroup of finite index, $L_\Ld$, $\Ld$ a lattice in $\R^n$ that is $B$-stable for each $BL_b \in \G$. The restriction to $\G$ of the canonical projection
$\text{I}(\R^n) \rightarrow \on$ given by $BL_b\mapsto B$ is a homomorphism with kernel $L_{\Ld}$ and its image $F$ is a finite subgroup of $\on$. Thus, we have an  exact sequence of groups
\begin{equation}\label{exseq}
0 \rightarrow \Lambda \rightarrow \G \rightarrow F \rightarrow 1.
\end{equation}
The group $F \simeq \Ld \backslash \G$ is called the {\em holonomy group} of $\G$.
The action by conjugation $BL_\ld B^{-1}=L_{B\ld}$ of $\Ld \backslash \G$ on $\Ld$ defines a representation $F \rightarrow \mathrm{GL}_n(\Z)$ called the {\em integral holonomy representation}, or, for short, the {\em holonomy representation}. In general, the integral holonomy representation is far from determining a flat manifold uniquely.

We note that in any compact flat $n$-manifold, we have that
\begin{equation*}\label{nB}
n_B := \dim \ker (B - \I_n) = \dim \, (\R^n)^B \ge 1
\end{equation*}
for every $BL_b \in \G$ (see for instance \cite{MR09}) and that $M_\G$ is orientable if and only if $F\subset \son$.

\subsection{$\Z_p$-manifolds}\label{S2-ZPM}
A \emph{$\Z_p$-manifold} is a compact flat manifold $M_\G = \G \backslash \R^n$ with holonomy group $F\simeq \Z_p$. Hence $\G = \langle BL_b, \Ld \rangle$ is torsion-free, with $B\in \on$ of order $p$ and $b\in \R^n \smallsetminus \Lambda$.

\sk
By \eqref{exseq}, $M_\G$ can be thought to be constructed from a $\Z_p$-action on $\Lambda$.
Thus, as a $\Z_p$-module, $\Ld$ is of the form given by Reiner in \cite{Re}, i.e.
$\Ld$ is isomorphic to
\begin{equation}\label{Reiner}
\Ld(a,b,c,\mathfrak a) := \mathfrak{a} \oplus (a-1)\, \mathcal{O} \oplus b \,\Z[\Z_p] \oplus c \, \I,
\end{equation}
where $a,b,c$ are non negative integers satisfying $a+b>0$ and
\begin{equation*}\label{dimension}
n=a(p-1)+bp+c,
\end{equation*}
$\xi$ is a primitive $p^{\mathrm{th}}$-root of unity, $\mathcal{O} = \Z[\xi]$ is the full ring of algebraic integers in the cyclotomic field $\Q(\xi)$ and
$\mathfrak{a}$ is an ideal in $\mathcal{O}$. Also, $\Z[\Z_p]$ denotes the group ring over ${\mathbb Z}$, and $\I \simeq
\Z$ stands for the trivial $\Z_p$-module.

\msk The $\Z_p$-actions on the modules $\mathcal{O}$, $\mathfrak{a}$ and $\Z[\Z_p]$ are given by multiplication by $\xi$.
In the bases $1,\xi,\ldots,\xi^{p-2}$ of $\mathcal{O}$ and $1,\xi,\ldots,\xi^{p-1}$ of $\Z[\Z_p]$, the actions of the generator are given, in matrix form,  respectively by
{\small
\begin{equation}\label{blocks}
C_p= \left( \begin{smallmatrix}
  0  &        &         &        &  -1     \\
  1  &  0     &         &        &  -1     \\
     &  1     &         &        &  -1     \\
     &        &  \ddots &        & \vdots  \\
     &        &         &    0   &  -1     \\
     &        &         &    1   &  -1
\end{smallmatrix} \right)  \in \mathrm{GL}_{p-1}(\Z),
\qquad
J_p= \left( \begin{smallmatrix}
         0  &        &         &            &  1  \\
         1  &  0     &         &            &  0  \\
            &  1     &         &            &  0   \\
            &        & \ddots  &            &  \vdots   \\
            &        &         &          0 &  0  \\
            &        &         &          1 &  0
\end{smallmatrix} \right) \in \mathrm{GL}_{p}(\Z).
\end{equation}}
If $\mathfrak{a} = \mathcal{O}\alpha$ is principal, we may use the $\Z$-basis $\alpha,\xi\alpha,\ldots,\xi^{p-2}\alpha$ of $\mathfrak{a}$ and the action of the generator is again described by the matrix $C_p$. For a general ideal this action is given by a more complicated integral matrix that we shall denote by $C_{p, \mathfrak{a}}$. We note that $C_{p, \mathfrak{a}}^p = \I$, the action has no fixed points, and the eigenvalues of $C_{p, \mathfrak{a}}$ are again all primitive $p^{\mathrm{th}}$-roots of 1. In particular, $C_{p, \mathfrak{a}}$  is conjugate to $C_p$ in $\mathrm{GL}(n, \Q)$.

Note that $J_p \in \mathrm{SO}(p)$ but $C_p \in \mathrm{SL}_{p-1}(\Z) \smallsetminus  \mathrm{O}(p-1),$ and furthermore, $n_{J_p}=1$, $n_{C_p}=n_{C_{p,a}}=0$. Since $\det \,C_p = \det \, J_p =1$ we have $F\subset \mathrm{SO}(n)$, and hence $M_\G$ is orientable.

\msk Using \eqref{Reiner}, Charlap was able to give a full classification of flat $\Z_p$-manifolds up to affine equivalence classes in \cite{Ch65} (see also \cite{Ch88}). He distinguished between two cases, that he called exceptional and non-exceptional manifolds.

\begin{definition}[\cite{Ch88}]
A $\Z_p$-manifold $M$ is called \textit{exceptional} if the lattice of translation is, as a $\Z_p$-module, isomorphic to $\Lambda(a,0,1,\mathfrak a)$ for some ideal $\mathfrak a$ in $\mathcal{O}=\Z[\xi]$; that is, if $(b,c)=(0,1)$. Otherwise, $M$ is called \textit{non-exceptional}.
\end{definition}

The following proposition collects several standard facts on the structure of $\Z_p$-manifolds. We include a sketch of the proof to make the paper more self-contained.

\begin{proposition}\label{propZp}
Let $M_\G = \G \backslash \R^n$ be a $\Z_p$-manifold with $\Gamma = \langle \g ,\, \Lambda \rangle$,
where $\g = BL_b$.

\sk (i) $(BL_b)^p= L_{b_p}$ where $b_p =  \sum_{j=0}^{p-1} B^j b \in L_\Ld \smallsetminus (\sum_{j=0}^{p-1} B^j ) \Lambda$.

\sk (ii)
As a $\Z_p$-module, $\Lambda\simeq \Ld(a,b,c,\mathfrak a)$ as in \eqref{Reiner} with $c\ge 1$ and $a,b,c$ uniquely determined by the isomorphism class of $\G$.

\sk (iii)
$\Gamma$ is  conjugate in $\mathrm{I}(\R^n)$ to a Bieberbach group $\tilde\Gamma = \langle \tilde\g ,\, \Lambda \rangle$
where
$\tilde\g=BL_{\tilde b}$ for which one further has that  $B\tilde b = \tilde b$ and $\tilde b
\in \frac 1p \Lambda \smallsetminus  \Ld$.

\sk (iv) One has that
$$H_1(M_\G,\Z) \simeq \Z_p^a \oplus \Z^{b+c}, \qquad H^1(M_\G,\Z) \simeq \Z^{b+c},$$ and hence $n_B = b+c = \beta_1$, where $\beta_1$ is the first Betti number of $M_\G$.

\sk (v) We have that $n_B = 1 \: \Leftrightarrow \: (b,c)=(0,1)$. In this case, there is a $\Z$-basis $f_1, \dots, f_n$ of $\Lambda$ such that $\Lambda_0 = \Z f_n$ and $\langle f_j, f_n \rangle=0$ for any $1\le j \le n-1$. Furthermore, the element $\g = BL_b$ as above can be chosen so that $b = \frac 1p f_n$.

\end{proposition}

\begin{proof}
(i) By repeatedly applying \eqref{semidirectproduct}, we get $(BL_b)^p= L_{\sum_{j=1}^{p} B^{-j}  b} \in \G$ and hence $b_p = \sum_{j=0}^{p-1} B^j b \in \Ld$. Furthermore, $b_p \not \in (\sum_{j=0}^{p-1} B^j) \Lambda$. In fact, if $b_p = \sum_{j=0}^{p-1} B^j \lambda$ with $\lambda \in \Lambda$, then we would have $(BL_p L_{-\lambda})^p = \I$, which contradicts the torsion freeness of $\G$.

\sk
(ii) By Charlap's classification \cite{Ch65} (see also \cite{Ch88}), the translation lattice $\Lambda$ must be one of the Reiner's $\Z_p$-module described in \eqref{Reiner}. Also, the torsion-free condition on $\G$ implies that $c\ge 1$. By (iv), $\G$ determines $a$ and $b+c$. Since $n = (a+b)(p-1)+ (b+c)$, the number $a+b$ is also determined and hence so are $b$ and $c$.

\sk
(iii)
If $BL_b \in \Gamma$ as in the statement, we have that $b = b_+ + b'$ where one has $Bb_+ = b_+$ and
$b'
\in
\ker (B-\I)^\perp$. Furthermore, $b_p'=0$ since $b_p'= (\sum_{j=0}^{p-1} B^j) b'$ lies in $\ker
(B-\I)^\perp \cap \ker (B-\I)$. Thus $(BL_b)^p = L_{p b_+}$ is a translation in $\Gamma$, hence
$pb_+ \in \Lambda$ and $b_+ \ne 0$.

\sk
If $v \in \R^n$, then  $L_v BL_b L_{-v} = BL_{b+(B^{-1} -\I)v}$. Now, we have
$\text{Im}(B^{-1} -\I)$ = $\ker (B^{-1}-\I)^\perp$ = $\ker (B-\I)^\perp$, so, one can choose  $v$  
so $(B^{-1} -\I)v = -b'$. In this way, conjugation of $\Gamma$ by $L_v$
changes $\Gamma$ into a Bieberbach group generated by $\tilde \g = BL_{b_+}$
and $\Lambda$, where $\tilde \g$ satisfies $Bb_+ = b_+$, $pb_+ \in \Lambda$ and
$b_+ \notin \Ld$, as desired.

\sk
(iv)
These groups are given in  \cite{Ch88}, pp.\@ 153, Exercise 7.1 (iv). For completeness, we  give a sketch of the proof by  explicit calculations. We note that the result for $H_1(M_\G,\Z)$ implies the one for $H^1(M_\G,\Z)$, by the universal coefficient theorem, and furthermore, the formula for $H^1(M_\G,\Z)$ in turn implies that $\beta_1 = b+c =n_B$.

\sk
Thus, it suffices to compute $H_1(M_\G,\Z) \simeq \G/[\G, \G]$. Since $[L_\ld,L_{\ld'}]=\I$ and $[\g,L_\ld] = B L_b L_\ld L_{-b} B^{-1} L_{-\ld} = L_{(B - \I)\ld}$, we have that
\begin{equation}\label{commutator}
[\G, \G] = \langle [\g,L_\ld] : \ld \in \Ld \rangle  =  L_{(B-\I)\Ld} \,.
\end{equation}

In order to compute $(B-\I)\Ld$  in our case we use  a basis of $\Ld$ such that the action of $B$ is represented by matrices as in \eqref{blocks}. We shall denote by $\Ld_ \mathcal{O}$ the sum of all submodules of $\Ld$ of type $\mathcal{O}$ or $\mathfrak{a}$, by $\Ld_R$ the sum of those of type  $\Z[\Z_p]$ and by $\Lambda_0$ the sum of the trivial submodules.

\sk
We first note that if $f_1, \dots, f_{p-1}$ is a $\Z$-basis of  a module $N$ of type $\mathcal{O}$ or $\mathfrak{a}$, then
a basis of $(B-\I)N$ is given by $f_2-f_1,\ldots,f_{p-1}-f_{p-2}, -\sum_{j=1}^{p-1} f_j - f_{p-1}$, or else
we can use the basis $f_2-f_1,f_3-f_2,\ldots,f_{p-1}-f_{p-2}, p f_{p-1}$. This implies that $N/(B-\I)N \simeq \Z_p$, hence $$\Ld_ \mathcal{O}/(B-\I)\Ld_ \mathcal{O} \simeq \Z_p^a.$$

Similarly,  if  $f'_1, \dots, f'_{p}$ is a $\Z$-basis of  a module $N'$ of type $\Z[\Z_p]$, then
a basis of $(B-\I)N'$ is given by $f'_2-f'_1,  \ldots , f'_{p-1}-f'_{p-2}, f'_p - f'_{p-1}, f_1' - f_{p}'$. This  implies that $N'/(B-\I)N'\simeq \Z$. Finally, for  a summand of trivial type, $N'' \simeq \Z$, we clearly have $(B-\I)N''=0$.
Thus $$\Ld_R \oplus \Ld_0 \, / \, (B-\I)\Ld_R \simeq \Z^{b+c}.$$

\sk

Now $(BL_b)^p= L_{pb_+}$ and $pb_+ \in \Ld$ is fixed by $B$, hence one has that
$pb_+
\in
\Ld^B = \Ld_R^B \oplus \Ld_0$, since
$(\Ld_ \mathcal{O})^B=0$ (a module of type $N$ has no $B$-fixed vectors).
For a module of type $N'$ we have that $(N')^B =\Z (\sum_{i=1}^p f'_j) \simeq \Z$,
by using a basis $f'_j,\, i\le j \le p,$ as above. Hence $\Ld_R^B \simeq \Z^b$.

Putting all this information together, by \eqref{commutator}, it is not hard to check that
$$\G/[\G, \G] \simeq \langle BL_b, L_ {\Ld_R^B \oplus \Ld_0} \rangle \, / \, (B-\I)\Ld_R\,\simeq \Z_p^a \oplus \Z^{b+c},$$
and hence the assertions in (iv) are proved.

\sk
(v) Since $n_B=b+c$ and $b\ge 0$, $c\ge 1$, it is clear that $n_B=1$ if and only if $(b,c)=(0,1)$.

\sk
Now, by (ii), we may assume that $b_+ = b$, after conjugation of $\Gamma$ by $L_v$ in $\mathrm{I}(\R^n)$ if necessary. By the description of the lattice in (ii), and since $(b,c) = (0,1)$, there is a $\Z$-basis $f_1, \dots, f_n$ of $\Lambda$ such that $\Lambda_0 = \Z f_n$ and $pb = af_n$ with $a \in \Z$, and where
$(p,a)=1$, since $b \notin \Lambda$. Now, if $s,t \in\Z$ are such that $sa+tp=1$, then  $spb = sa f_n = f_n - tp f_n$, so $sb = \frac {f_n}p - t f_n$. Hence, since $tf_n \in \Lambda$, we may change the generator $\gamma$, of $\G$ mod $\Ld$, by $\tilde \g := (BL_b)^s L_{t f_n}= B^s L_{\f 1p f_n}$, which has the asserted properties. Finally, we note that since $B$ has no fixed points on $(\Ld_0)^\perp$ then $\langle f_j, f_n \rangle=0$ for any $1\le j \le n-1$.

This completes the proof.
\end{proof}

\begin{remark} \label{cfm dim3}
The compact flat manifolds are classified, up to affine equivalence, only in low dimensions $n\le 6$ (\cite{HW} $n\le 3$, \cite{BBNWZ} $n=4$ and \cite{CZ} $n=5,6$). In dimension 3 there are 10 compact flat manifolds, half of them having cyclic holonomy groups $F\simeq \Z_2,\Z_3,\Z_4, \Z_6$ (see \cite{Wo}). These manifolds were described in \cite{CR}, where they were called \textit{platycosms}. Out of these, there is only one $\Z_3$-manifold, the \textit{tricosm}. That is, there is only one 3-dimensional $\Z_p$-manifold with $p$ an odd prime. It is denoted by $\mathcal{G}_3$ in \cite{Wo} and by $c3$ in \cite{CR}.
\end{remark}

\subsection{The models $\zp$}
\label{zpabc}
We shall now  give some explicit models of $\Z_p$-manifolds. In particular, we will show that for any triple, $a,b,c\in \N_0$, $c\ge 1$ and any ideal $\mathfrak{a}$ in $\mathcal {O}$, one can construct a compact flat manifold with  translation lattice $\Ld$ such that, as a $\Z_p$-module it satisfies \eqref{Reiner}. For $a,b,c \in \N_0$,  $c\ge 1$, and any ideal $\mathfrak{a}$, let $C_p, J_p$, and $C_{p, \mathfrak{a}}$ be as in the previous subsection and define matrices $C, C' \in \GL_n(\Z)$ of the form
\begin{equation}\label{matrices C}
\begin{split}
& C = \text{diag}(\underbrace{C_p,\ldots,C_p}_{a},\underbrace{J_p,\ldots,J_p}_{b},\underbrace{1,\ldots,1}_{c})\,,
\\
& C' = \text{diag}(C_{p, \mathfrak{a}}, \underbrace{C_p,\ldots,C_p}_{a-1},  \underbrace{J_p,\ldots,J_p}_{b},
\underbrace{1,\ldots,1}_{c})\,,
\end{split}
\end{equation}
where $\mathfrak{a}$ is not principal.

Note that, actually, $C$ is just a particular case of $C'$ when $\mathfrak{a}= \mathcal{O}$ (or when $\mathfrak{a}$ is principal) and $C'$ depends on $a,b,c$ and $\mathfrak{a}$, though this is not apparent in the notation. Although  $C, C' \not \in \mathrm{O}(n)$, they can be conjugated into  $\text{I}(\R^n)$. Indeed, the eigenvalues of
$C_p$ and $C_{p, \mathfrak{a}}$ are exactly the primitive $p^{\mathrm{th}}$-roots of unity and the eigenvalues of $J_p$ are all $p^{\mathrm{th}}$-roots of unity. Thus, if $\sim$ means conjugation in $\GL_n(\R)$, then $C_p \sim B_p$,  $C_{p, \mathfrak{a}} \sim B_p$ and
$J_p \sim (\begin{smallmatrix} B_p & \\ & 1 \end{smallmatrix})$, where
\begin{equation}\label{Bp-matrix}
B_p := \text{diag}\big( B(\tf{2\pi}{p}), B(\tf{2\cdot 2\pi}{p}), \ldots,
B(\tf{2q\pi}{p})\big),
\end{equation}
with $q=[\f{p-1}2]$ and $B(t) = \left( \begin{smallmatrix} \cos t & -\sin t
\\ \sin t & \cos t \end{smallmatrix}\right)$, $t \in \R$.
That is, there exists a matrix $X_{\mathfrak a} \in \GL_{n-c}(\R) \subset \GL_n(\R)$ such that
$X_{\mathfrak a} C' X^{-1}_{\mathfrak a} = B \in \text{SO}(n-c) \subset \text{SO}(n)$, where
\begin{equation}\label{diag Bp-matrix}
B = \text{diag}(\underbrace{B_p,\ldots,B_p}_{a+b},\underbrace{1,\ldots,1}_{b+c}).
\end{equation}

We now define a lattice in $\R^n$ by
$$\Ld_{p,a}^{b,c}(\mathfrak{a}) := X_{\mathfrak a} \Z^{n-c} \stackrel{\perp}{\oplus} \Z^{c} = X_{\mathfrak{a}} L_{\Z^n} X_{\mathfrak{a}}^{-1}.$$
Thus, as a $\Z_p$-module, $\Ld_{p,a}^{b,c}(\mathfrak{a})$ decomposes as in \eqref{Reiner},
with $c\I$ orthonormal to its complement $\mathfrak{a} \oplus (a-1)\mathcal{O}\oplus  b\Z[\Z_p]$. With these ingredients, we define an $n$-dimensional Bieberbach group:
\begin{equation*} \label{Gpa}
\G_{p,a}^{b,c} (\mathfrak a) :=  \langle BL_{\tf{e_n}{p}}, \Ld_{p,a}^{b,c}(\mathfrak{a}) \rangle,
\end{equation*}
where $e_n$ is the canonical vector, and the corresponding flat Riemannian $n$-manifold
\begin{equation} \label{mzp}
M_{p,a}^{b,c}  (\mathfrak a):= \G_{p,a}^{b,c}  (\mathfrak a) \backslash \R^n.
\end{equation}

\begin{remark}\label{essential}
As we shall see, in the study of eta invariants 
of $\Z_p$-manifolds it will essentially suffice to look at exceptional $\Z_p$-manifolds, i.e.\@ those with $\beta_1 =1$. Proposition \ref{propZp} (vi) says that any  exceptional $\Z_p$-manifold $M$ is diffeomorphic to some $M_{p,a}^{0,1}(\mathfrak a)$ as in \eqref{mzp}, i.e. having $b = \tf1p e_n$.
\end{remark}

As mentioned in the Introduction, the integrality result in Theorem \ref{twisted etas mod Z} will be proved to hold for every $\Z_p$-manifold except for a single one, the so called tricosm. We now give a description of this manifold.

\begin{example}\label{tricosm}
In our previous description, the tricosm $c3$ (see Remark \ref{cfm dim3}) corresponds to $M_{3,1} = M_{3,1}^{0,1}(\mathcal{O})$, with $\mathcal{O} = \Z[\f{2\pi i}{3}]$.
So, as a $\Z_3$-module, we have $\Lambda = \Z[e^{\f{2\pi i}{3}}] \oplus \Z$ and the integral representation of $F\simeq \Z_3$ is given by the matrix
$\tilde C_3 = \mathrm{diag}(C_3,1) =
\left( \begin{smallmatrix} 0 & -1 & \\ 1 & -1 & \\ & & 1 \end{smallmatrix} \right)$. Then
$M_{3,1} = \langle BL_{\f{e_3}{3}}, L_{f_1}, L_{f_2}, L_{e_3} \rangle \backslash \R^3$ where $f_1,f_2,e_3$ is a $\Z$-basis of $\Lambda_{3,1} = X\Z^2 \oplus \Z$, $B = \left( \begin{smallmatrix} -1/2 & -\sqrt 3/2 & \\ \sqrt 3/2 & -1/2 & \\ & & 1 \end{smallmatrix} \right) \in \mathrm{SO}(3)$ and $X\in \mathrm{GL}_3(\R)$ is such that $X \tilde C_3 X^{-1} = B$.
\end{example}

\subsection{Spin group and spin representations}\label{S2-SGSR}
Let $Cl(n)$ denote the Clifford algebra of $\R^n$ with respect to the standard inner product. Inside the group of units of $Cl(n)$ there is the spin group, $\text{Spin}(n)$, which is a compact, simply connected Lie group
if $n\geq 3$. There is a canonical epimorphism
\begin{equation}\label{picovering}
\pi:\text{Spin}(n)\rightarrow \text{SO}(n)
\end{equation}
given by $\pi(v)(x) = vxv^{-1}$, with kernel $\{\pm 1\}$.
A maximal torus of $\text{Spin}(n)$ has the form
$T=\left\{x(t_1,\ldots,t_m) \,:\, t_j \in \R, \, 1\le j\le m \right\}$
where $m = [\frac n2]$,
\begin{equation*}\label{xelements}
x(t_1,\dots,t_m) := \prod_{j=1}^m (\cos t_j + \sin t_j \, e_{2j-1}e_{2j})
\in \text{Spin}(n)
\end{equation*}
and $e_1,\ldots,e_n$ is the canonical basis of $\R^n$. For convenience, if $a\in \N$, we shall use the notation
\begin{equation} \label{notacionxj}
x_{a}(t_1,t_2,\ldots,t_q) := x(\underbrace{t_1,t_2,\ldots,t_q}_{1}, \ldots,
\underbrace{t_1,t_2,\ldots,t_q}_{a}) \,.
\end{equation}

Set $x_0(t_1,\ldots,t_m):= \mathrm{diag}(B(t_1),\ldots, B(t_m),1)$ if
$n=2m+1$ and omit the 1 if $n=2m$. A maximal torus in $\text{SO}(n)$ is 
$T_0=\left\{x_0(t_1,\dots,t_m) \;:\; t_i \in \R \right\}$.
The restriction map $\pi:{T}\rightarrow T_0$ is a 2-fold covering and
\begin{equation}\label{2 fold}
\pi(x(t_1,\dots,t_m)) = x_0(2t_1,\ldots,2t_m) \,.
\end{equation}

Let $(L_n,\s_n)$ denote the {\em spin representation} of $\spin$, which is
the restriction to $\spin$ of any irreducible representation of the complex
Clifford algebra $Cl(n)\otimes \C$. It has complex dimension $2^m$ with
$m=[\tfrac n2]$. If $n$ is odd,
$(L_n,\s_n)$ is irreducible while, if $n$ is even, $\s_n = \text{S}_n^+
\oplus \text{S}_n^-$ where each $\s_n^\pm$ is irreducible of
dimension $2^{m-1}$. The representations  $L_n^\pm := {L_n}_{|S_n^{\pm}}$
are called the {\em half-spin representations}. It is known
that the values of the characters of $L_n$ and $L^\pm_n$ on the torus $T$
are given by (see \cite{MiPo06}, Lemma 6.1)
\begin{equation}\label{spinpmchars}
\begin{split}
& \chi_{_{L_n}} (x(t_1,\dots,t_m)) = 2^{m} \prod_{j=1}^m \cos t_j, \\
& \chi_{_{L^\pm_n}} (x(t_1,\dots,t_m))= 2^{m-1}\Big( \prod_{j=1}^m \cos t_j
\pm i^m \prod_{j=1}^m \sin t_j \Big).
\end{split}
\end{equation}

\subsection{Spin structures}\label{S2-SpSt}
It is a well known fact that spin structures on a compact flat spin manifold $M_\G$ are in a 1--1 
correspondence with group homomorphisms
\begin{equation}\label{spindiagram}
\vep:\G\arr \spin \qquad  \text{ such that } \qquad \pi (\vep(\gamma)) = B,
\end{equation}
for any $\g=BL_b \in \G$ (see \cite{Fr97}, \cite{LM} or \cite{Pf}), where $\pi$ is as in (\ref{picovering}).

\sk
Any morphism $\vep$ as in \eqref{spindiagram} is determined by 
the generators of $\G$. Let $M_\G$ be a $\Z_p$-manifold with $\G = \langle \g, L_\Ld \rangle$ and
let $f_1,\ldots,f_n$ be a $\Z$-basis of $\Lambda$.
Since  $r(\Ld)=\I$, we have $\vep(\Ld) \in \{\pm 1\}$, and hence $\vep$ is determined by $\vep(\g)$ and
\begin{equation*}\label{delta_j}
\delta_j := \vep(L_{f_j}) \in \{\pm 1\}, \qquad j=1,\ldots,n.
\end{equation*}

Since every $F$-manifold with $|F|$ odd is spin (\cite{Va}, Corollary 1.3), the $\Z_p$-manifolds considered are spin.
The determination of such structures for $\Z_p$-manifolds were previously considered in some particular cases (see \cite{MiPo06} and \cite{MiPo09} for the exceptional manifolds $M_{p,a}^{0,1}(\mathfrak a)$, and in \cite{SS} in the case of $M_{p,1}^{0,1}$ and $p$ any odd integer, not necessarily prime).
In fact, it is known that the exceptional manifolds $M_{p,a}^{0,1}(\mathfrak a)$ admit only two spin structures, $\vep_1, \vep_2$, one of which,  $\vep_1$, is of trivial type (\cite{MiPo09}). They are given, for $h=1,2$, by
\begin{equation} \label{eq.spinstructuresX}
\begin{split}
& \vep_h(L_{f_1}) = \cdots = \vep_h(L_{f_{n-1}}) = 1, \qquad \vep_h(L_{f_n}) = (-1)^{h+1},
\msk \\
& \vep_h(\g) = (-1)^{a[\f{q+1}{2}] + h + 1} \; x_a\big(\tf{\pi}p,\tf{2\pi}p,\dots, \tf{q\pi}p \big),
\end{split}
\end{equation}
in the notation of \eqref{notacionxj}.

\sk
Although in the sequel we will not need the explicit description of the spin structures of an arbitrary $\Z_p$-manifold, we will now give it for completeness. To this end, we will use the following abuse of notation
\begin{equation}\label{notation}
\vep_{|\Ld} = (\vep(L_{f_1}),\ldots,\vep(L_{f_n})).
\end{equation}

\begin{proposition}\label{spinstructs}
Let $p$ be an odd prime and let $M$ be a $\Z_p$-manifold with lattice of translations $\Ld \simeq \Ld(a,b,c,\mathfrak a)$. Then $M$ admits exactly $2^{b+c}=2^{\beta_1}$ spin structures, only one of which is of trivial type.

If, in particular, $M=\zp$ then the spin structures are explicitly given by
\begin{equation}\label{eq.spinstructures}
\begin{split}
& {\vep}_{|\Ld} = 
\big( \underbrace{1,\ldots,1}_{a(p-1)}, \underbrace{\delta_1, \ldots,\delta_1}_p,
\ldots, \underbrace{\delta_b,\ldots,\delta_b}_p, \delta_{b+1},\ldots,\delta_{b+c-1},
(-1)^{h+1} \big) \msk \\
& \vep(\g) 
= (-1)^{(a+b)[\f{q+1}{2}] + h + 1}
\; x_{a+b}\big(\tf{\pi}p, \tf{2\pi}p,\dots, \tf{q\pi}p \big),
\end{split}
\end{equation}
in the notations of \eqref{notation}, \eqref{delta_j} and \eqref{notacionxj}.
\end{proposition}

\begin{proof}
Let $M=\G\backslash \R^n$ be a $\Z_p$-manifold. By the results in \cite{MiPo06} (see also \cite{MiPo09}, \cite{MiPo04}), a group homomorphism $\vep : \G \rightarrow \spin$ as in \eqref{spindiagram} determines a spin structure on $M$ if and only if it satisfies the following conditions:
\begin{itemize} 
\sk
\item[\textsf{C1}.] \qquad $\vep(L_{B\ld})=\vep(L_\ld) \quad \text{ for any } \ld \in \Ld$,

\msk
\item[\textsf{C2}.] \qquad $\vep(\g)^p = \vep(\g^p) = \vep(L_{pb_+})$,
\end{itemize}
where $\g=BL_{b}$ with $b=b_+ + b'$ and $b' \perp b_+$. Furthermore, the orthogonal projection of $b_+$ on $\Ld_0 = c\I$ is not $0$.

\sk We fix a $\Z$-basis $f_1,\ldots,f_{n-1},f_n=e_n$ of $\Ld$.

\sk \noindent
\emph{Case 1, $M \simeq \zp$}
First, suppose that $M = \zp$ and assume a group homomorphism $\vep :\G_{p,a}^{b,c}(\mathfrak a) \rightarrow \spin$ as in \eqref{spindiagram} is given.

By \eqref{Bp-matrix}, \eqref{diag Bp-matrix} and \eqref{2 fold} we have
\begin{equation} \label{vep(g)}
\vep(\g) = \pm (-1)^{(a+b)[\f{q+1}{2}]} \; x_{a+b}\big(\tf{\pi}p, \tf{2\pi}p,\dots, \tf{q\pi}p \big)
\end{equation}
where $q=\f{p-1}{2}$.

We note that in this case, since $b_+ = \f {e_n}{p}$, condition \textsf{C2} will only give a condition for $\vep$ on $\Z e_n$. To determine the action of $\vep$ on
$$(\Z e_n)^\perp = \Ld_\mathcal{O} \oplus \Ld_R \oplus \Ld_0' = \mathcal{O}^{\oplus a} \oplus R^{\oplus b} \oplus \I^{\oplus (c-1)},$$
we will use condition $\textsf{C1}$ together with the integral matrix $C'$ given in \eqref{matrices C}. Let $\tilde \G := \langle C'L_{\f{e_n}{p}}, \tilde \Ld \rangle \subset \mathrm{Aff}(\R^n)$ where $\tilde \Ld = \Ld(a,b,c,\mathfrak a)$ is as in \eqref{Reiner}. By the description in \S \ref{zpabc}, we have $X_{\mathfrak a} \tilde \G X_{\mathfrak a}^{-1} = \G_{p,a}^{b,c}(\mathfrak a)$.
Now, define $\tilde \vep: \tilde \G \arr \spin$ by $\tilde\vep = \vep \circ I_{X_{\mathfrak a}}$ where $I_{X_{\mathfrak a}}$ is conjugation by $X_{\mathfrak a}$. Since $$\vep(L_{(B-\I)\Ld_{p,a}^{b,c}(\mathfrak a)}) = \vep( X_{\mathfrak a} L_{(C'-\I) \Ld} X_{\mathfrak a}^{-1}) = \tilde\vep(L_{(C'-\I)\Ld})$$ we have that $\vep(L_{(B-\I)\Ld_{p,a}^{b,c}(\mathfrak a)})=1$ if and only if $\tilde\vep(L_{(C'-\I)\Ld}) = 1$.

\sk \noindent
\emph{Step 1.} 
Here we will show that $\vep_{|\Ld_\mathcal{O}}\equiv 1$. For any summand of type $\mathcal{O}=\Z[\xi]$ in \eqref{Reiner}, there is a $\Z$-basis of the form $\{e,\xi\,e,\ldots,\xi^{p-2}\,e\}$. Hence by condition \textsf{C1} we must have
$1 = \tilde\vep(\xi\,e - e)=\cdots = \tilde\vep(\xi^{p-2}\, e - \xi^{p-3}\,e) = \tilde\vep(\xi^{p-1}\, e - \xi^{p-2}\,e)$.
Thus, we have
$\tilde\vep(e) = \tilde\vep( \xi\,e)= \cdots = \tilde\vep(\xi^{p-2}\,e) =  
\tilde\vep(e) \, \tilde\vep (\xi\,e) \cdots \tilde\vep (\xi^{p-2}\,e)$,
where in the last equality we have used that $\xi^{p-1}=-\sum_{i=0}^{p-2} \xi^i$.
This implies $\tilde\vep(e)^{p-2} = 1$, and hence $\tilde\vep(e)=1$ since $p$ is odd. Therefore, $\tilde\vep(\xi^j\, e)=1$ for every
$0\le j \le p-2$ and thus $\tilde\vep (\lambda )=1$ for any $\lambda \in \Ld_\mathcal{O}$.

Now, given a summand of type $\mathfrak{a}$ in $\Ld$, there exist $e_1$, $e_2 \in \mathfrak{a}$ such that $\mathfrak{a} = \mathcal{O}e_1 + \mathcal{O}e_2$.
By the same argument as in the case of $\mathcal{O}$, we conclude that $\tilde\vep_{|\mathcal{O}e_1 } = \tilde\vep_{|\mathcal{O}e_2 }
=1$. Hence $\tilde\vep_{|\mathfrak{a}}=1$.

\sk
In this way, for any $\lambda \in \Ld_\mathcal{O}$ we have $$\vep(L_{\ld}) = \vep(L_{X_{\mathfrak a} \ld}) = \vep(X_{\mathfrak a} L_{\ld} X_{\mathfrak a}^{-1}) = \tilde\vep(L_{\ld}) =1.$$
In particular, $\delta_j = \vep(L_{f_j}) = 1$ for $1\le j \le a(p-1)$.

\sk \noindent
\emph{Step 2.}
We now study the action of $\vep$ on the block $\Ld_R \oplus \Ld_0$. For each summand of type $R=\Z[\Z_p]$ in \eqref{Reiner}, there is a $\Z$-basis $\{e,\xi\,e,\ldots,\xi^{p-1}\,e\}$. Thus, by condition \textsf{C1} we have
$\tilde\vep(e) = \tilde\vep( \xi \,e) = \cdots = \tilde\vep(\xi^{p-1}\,e) \in \{\pm 1 \}$
since there are no extra restrictions. Hence $\tilde \vep_{|R} = \delta$, where $\delta\in \{\pm 1\}$ and, proceeding as before, we have
$$\vep(L_{f_{a(p-1)+jp+1}}) = \cdots = \vep(L_{f_{a(p-1)+jp+p}}) = \delta_j, \qquad 0 \le j \le b-1.$$
Also, for trivial summands, it is clear that $\vep(L_{f_i}) = \delta_i$ for $n-c+1 \le i \le n$. For $i\ne n$ there are no restrictions.

\sk \noindent
\emph{Step 3.}
Since $\g^p = L_{e_n}$, conditions \textsf{C1} and \textsf{C2} are linked together and give a restriction which determines both the value of $\vep(L_{e_n})$ and the sign in \eqref{vep(g)}.
Indeed, since $\g^p=L_{e_n}$ we have
$$\delta_n = 
\vep(\g)^p = \sigma \, x_{a+b}(\tf{\pi}p,\tf{2\pi}p,\ldots,\tf{q\pi}p)^p
= \sigma \, x_{a+b}(\pi,2\pi,\ldots,q\pi) = \sigma (-1)^{t+1}$$
with $\sigma = \pm (-1)^{(a+b)[\f{q+1}{2}]}$ and where we have used that $x(\theta)^k = x(k\theta)$ for any $\theta\in \R$, $k\in\Z$ and the commutativity in $\C l(n)$ of the elements $e_{2i-1}e_{2i}$ and $e_{2j-1}e_{2j}$ for $i \not = j$.

Putting $\pm = (-1)^{h+1}$ with $h=1,2$, we get the expressions in \eqref{eq.spinstructures}.

\sk \noindent \emph{Case 2, $M$ arbitrary}.
The proof is entirely analogous, where we now start with the $b+c-1$ characters $\delta_j$ corresponding to $\Ld_R$ and $\Ld_0'=(c-1)\I$, and where we have $(BL_b)^p =L_{pb_+}$ with $pb_+ \in \Ld_R^B \oplus \Ld_0$ (in place of $pb_+ =f_n$).

Then the equation  $\vep (\g)^p = \vep (\g^p) = \vep(L_{p b_+})$ imposes a condition linking the $\delta$'s
and we again have $2^{b+c}$ spin structures as in the case of the model before.
\end{proof}

\begin{remark}
It is known that if a manifold $M$ is spin, the inequivalent spin structures are classified by $H^1(M,\Z_2)$ (\cite{LM}, \cite{Fr97}). For $M$ a $\Z_p$-manifold, by Proposition \ref{propZp} (iv) and the universal coefficients theorem (or also directly), one can prove that $H^1(M,\Z_2) \simeq H_1(M,\Z_2) \simeq \Z_2^{b+c}$.
Hence, the number of spin structures of a $\Z_p$-manifold is $2^{b+c}=2^{\beta_1}$. In Proposition \ref{spinstructs} we give a direct proof of this fact together with an explicit description of these structures in the case of the models $M=\zp$.
\end{remark}


\section{Twisted eta series}\label{sect-3}

\subsection{Spectrum of twisted Dirac operators}\label{S2-STDO}
Let $(M_\G,\vep)$ be a  compact flat spin $n$-manifold with lattice of
translations $\Lambda$. Let $\rho : \Gamma \rightarrow U(V)$ be a unitary
representation such that $\rho_{|\Lambda} =1$. Consider the \textit{twisted
Dirac operator}
$$D_\rho = \sum_{i=1}^n L_n(e_i) \, \f{\partial}{\partial x_i},$$ where
$\{e_1,\ldots,e_n\}$ is an orthonormal basis of $\R^n$   and $L_n$ is
the spin representation, acting on smooth sections of the twisted spinor
bundle
$$\mathcal{S}_\rho(\man,\vep) = \G \backslash(\R^n\times (\s\otimes V))\arr \G\backslash \R^n$$
of $\man$ (see \cite{MiPo06} for details).

\sk Let $\Ld^\ast_{\vep} = \{ u \in \Ld^\ast : \vep(L_\ld) = e^{2\pi i
\ld\cdot u} \text{ for any } \ld \in \Ld \}$, where $\Ld^*$ is the dual
lattice. The nonzero eigenvalues of $D_\rho$ are of the form $\pm 2\pi \mu$
with $\mu = ||v||$ for some $v \in \Ld_\vep^*$. In \cite{MiPo06}, Theorem
2.5, it is shown that the multiplicities $d_{\rho,\mu}^\pm$ of $\pm 2\pi
\mu$ for $(\man,\vep)$ are given, for $n$ odd, by
\begin{equation}\label{eq.multipodd}
d_{\rho,{\mu}}^\pm(\G,\vep) = \tf{1}{|F|} \sum_{\g=BL_b \in \Ld \backslash
\G} \chi_\rho(\g) \sum _{u \in (\Ld_{\vep,\mu}^\ast)^B}
e^{-2\pi i u\cdot b} \;\chi_{_{L_{n-1}^{\pm \sigma(u,x_\g)}}}(x_\g).
\end{equation}
Here, $\chi_\rho$ and $\chi_{L_{n-1}^\pm}$ are the characters of $\rho$ and
of the half spin representations, respectively, and for $\g = BL_b \in \G$
we have $\Ld_{\vep,\mu}^\ast = \{ v \in \Ld_\vep^* \,:\, ||v||=\mu \}$ and
\begin{equation*}\label{lattice emuB}
(\Ld_{\vep,\mu}^\ast)^B = \{ v \in \Ld_{\vep,\mu}^* \,:\, Bv=v \}.
\end{equation*}
Furthermore, $x_\g \in T$ is a fixed element in the maximal torus of
$\spin$, conjugate in $\spin$ to $\vep(\g)$, and $\sigma(u, x_\g)$ is a sign
depending on $u$ and on the conjugacy class of $x_\g$ in $\text{Spin}(n-1)$
(see Definition 2.3 in \cite{MiPo06}).

\sk Relative to the multiplicity of the 0 eigenvalue, i.e.\@ the dimension of the space of harmonic spinors,
it is shown in \cite{MiPo06} that
\begin{equation}   \label{eq.harmonicspinors}
d_{\rho,0}(\Gamma,\vep) = \left\{ \begin{array}{ll} \tf{1}{|F|} \,
\sum\limits_{\g  \in \Ld \backslash \G} \, \chi_{_\rho}(\g) \;
{\chi}_{_{L_{n}}}(\vep(\g)) & \qquad  \text{if } \vep_{|\Ld}= 1, \sk \\  0 &
\qquad \text{if } \vep_{|\Ld} \ne 1.
\end{array} \right.
\end{equation}

\subsection{Spectral asymmetry}\label{specasymm}
Consider an arbitrary $\Z_p$-man\-ifold $M$ of dimension $n$ as in
\eqref{mzp}, with $p$ an odd prime, equipped with a spin structure
$\vep$. The formula for the multiplicity of the
eigenvalues \eqref{eq.multipodd} involves the character
$\chi_{\rho}$ of a representation $\rho : \Z_p \rightarrow U(V)$.
Thus, we will consider for each $0\le \ell \le p-1$, the Dirac
operator $D_\ell$ twisted by the characters
\begin{equation*}\label{eq.characters}
\rho_\ell : \Z_p \rightarrow \mathbb{S}^1 \subset \C^*, \qquad k \mapsto
e^{\f{2\pi ik \ell }{p}}.
\end{equation*}

\begin{remark}\label{symmetric spectrum}
By Corollary 2.6 in \cite{MiPo06}, valid for arbitrary compact flat manifolds, if $n$ is even, or else, if $n$ is odd and $n_B\ge 2$ for every $BL_b\in \G$, then the spectrum  of $D_\ell$ is symmetric. That is, one has that $d_{\ell,\mu}^+(\G,\vep) = d_{\ell,\mu}^- (\G,\vep)$. Hence $\eta_{\ell,\vep}(s) \equiv 0$ and, in particular, $\eta_{\ell,\vep} = 0$. In the case of $\Z_p$-manifolds, since $n_B= b+c$ and $c\ge 1$, we see that for the non-exceptional ones, i.e.\@ those with $(b,c) \ne (0,1)$, the 
eta invariant is just given by 
$\bar \eta_{\ell,\vep} = \tf 12 \dim \ker D_\ell$. This computation will be done in the next section (see \eqref{harmonics}).
\end{remark}

\sk
By the previous remark and Remark \ref{essential}, in the computations of $\eta_\ell(s)$ and $\eta_\ell$, we will need only consider \textsl{exceptional} $\Z_p$-manifolds of the form $M_{p,a}^{0,1}(\mathfrak a)$.

\msk
By \eqref{etaseries}, we can write
\begin{equation}\label{eq.etaseriesdif}
\eta_{\ell,h}(s):= \eta_{\ell}(\G,\vep_h)(s) =  \sum_{\pm 2\pi \mu \in
\mathcal{A}_{\ell,h}} \f{ d_{\ell,\mu,h}^+ - d_{\ell,\mu,h}^-}{(2\pi \mu)^s}
\end{equation}
for $\mathrm{Re}(s)>n$, where $d_{\ell,\mu,h}^\pm$ stand for
$d_{\rho_\ell,\mu}^\pm(\G,\vep_h)$ as given in (\ref{eq.multipodd}) and
$\mathcal{A}_{\ell,h}$ denotes the \emph{asymmetric spectrum}, that is
$$\mathcal{A}_{\ell,h} = \{ \pm 2\pi \mu \in Spec_{D_{\ell,h}}(M) : d_{\ell,\mu, h}^+\ne d_{\ell,\mu,h}^-\}.$$

To this end, we will first compute the differences,
$d_{\ell,\mu,h}^+ - d_{\ell,\mu,h}^-$
in Proposition~\ref{prop.Deltas}, and then the series in
\eqref{eq.etaseriesdif}, in Theorem \ref{thm.eta series}.

\begin{lemma}\label{lem.Deltafinal}
Let $p$ be an odd prime and $\ell \in \N$ with $0\le \ell \le p-1$. Let $M$ be  an exceptional $\Z_p$-manifold with a spin structure $\vep_h$, $h=1,2$. Then
\begin{eqnarray*}
&&d_{\ell,\mu,h}^+ - d_{\ell,\mu,h}^- \\
&=& (-1)^{(\f{p^2-1}{8})a+1} \, i^{m+1} \, 2 \, p^{\f{a}{2}-1} \, \sum_{k=1}^{p-1} (-1)^{k(h+1)} \, \big(\tf kp \big)^a
\, e^{\f{2\pi i k\ell}{p}}   \sin(\tf{2\pi \mu k}p),
\end{eqnarray*}
where $\big( \f{\cdot}p \big)$ is the Legendre symbol and $d_{\ell,\mu,h}^\pm$ denotes the multiplicity of the eigenvalue $\pm 2\pi \mu$ of $D_\ell$.
\end{lemma}

\begin{proof}
Given an exceptional $\Z_p$-manifold $M_\G$, by Proposition \ref{propZp} (iii), we may assume that $\G =\langle \g, \Ld \rangle$ with $\g = BL_b \in \G$, $B^p=\I$ and $b=\f{e_n}{p}$.

\sk
We begin by computing the expression in \eqref{eq.multipodd}.  Note that the holonomy group is $F \simeq \{\I,B,B^2,\ldots,B^{p-1}\}$. Since $\vep_h(\g^k) = \vep_h(\g)^k \in T$, 
we can take $x_{\g^k} = \vep_h(\g^k)$ and hence
$\sigma(e_n,x_{\g^k})=1$ for every $1\le k \le n-1$, by the definition of
$\sigma$ (see \cite{MiPo06}). Hence, according to (\ref{eq.multipodd}), if
$b_k$ is defined by the relation $\g^k = B^k L_{b_k}$, we obtain
\begin{equation}\label{multipoddzp}
d_{\ell,\mu,h}^\pm = \tf 1p \, \sum_{k=0}^{p-1} \rho_\ell(k) \sum _{u \in (\Ld_{\vep_h,\mu}^\ast)^{B^k}} e^{-2\pi i u \cdot b_k} \; \chi_{{L_{n-1}^{\pm \sigma
}}}(\vep_h({\g}^k))
\end{equation}
where $\sigma = \sigma(u,x_{\g^k})$.

Now, since $\Ld = \big( \mathfrak a \oplus (a-1)\mathcal{O}  \big) \stackrel{\perp}{\oplus} \Z e_n$
and $(\R^n)^{B^k} = \R e_n$, $1\le k \le n-1$, we have that $(\Ld_{\vep_h}^\ast)^{B^k} = \Z e_n$ if $h=1$ and $(\Ld_{\vep_h}^\ast)^{B^k} = (\Z+\tf12) e_n$ if $h=2$. Hence, 
\begin{equation}\label{eq.latis}
(\Ld_{\vep_h,\mu}^\ast)^{B^k} = \{\pm \mu e_n\}
\end{equation}
with $\mu \in \N$ for $\vep_1$ and $\mu \in \N_0+\tf12$ for $\vep_2$.

\sk In this way, using
\eqref{eq.latis} and the fact that $b_k = \f{k}{p}e_n$, we see that
\eqref{multipoddzp} reduces to
\begin{equation}\label{eq.multips}
d_{\ell,\mu,h}^\pm = \tfrac 1p \,\Big(2^{m-1}|\Ld_{\vep_h,\mu}^*| + \sum_{k=1}^{p-1} e^{\f{2\pi
i k\ell}{p}} S_{\mu,h}^\pm(k)
\Big)
\end{equation}
where we have put
\begin{equation}\label{eq.skmus}
S_{\mu,h}^\pm(k) := e^{\f{-2 \pi i \mu k}p} \chi_{L_{n-1}^\pm}(\vep_h(\g^k))
+ e^{\f{2 \pi i \mu k}p} \chi_{L_{n-1}^\mp}(\vep_h(\g^k)).
\end{equation}
Here we have used that $\sigma(-u,\gamma)=-\sigma(u,\gamma)$ and that
$\sigma(e_n,\g^k)=1$.

\sk Now, using that $x(\theta)^k=x(k\theta)$ for $\theta\in \R, k\in \Z$, we
have that
\begin{equation}\label{eq.spinstructure gk}
\vep_h(\g^k)= (-1)^{s_{h,k}} \; x_{a}\big(\tf{k\pi}p, \tf{2k\pi}p,\dots,
\tf{qk\pi}p \big)
\end{equation}
for $1\le k\le p$, where
\begin{equation}\label{eq.signs}
s_{h,k} := k([\tf{q+1}{2}]a + h + 1).
\end{equation}
Thus, by \eqref{eq.spinstructure gk} and using \eqref{spinpmchars}, we
obtain
$$\chi_{_{L^\pm_{n-1}}}(\vep_h(\g^k)) = (-1)^{s_{h,k}} \, 2^{m-1} \Big\{
\Big(\prod_{j=1}^q \cos(\tf{jk\pi}p)\Big)^{a} \pm i^m \Big(\prod_{j=1}^q
\sin(\tf{jk\pi}p)\Big)^{a} \Big\}.$$
Substituting this expression into \eqref{eq.skmus} we see that
\begin{multline}\label{eq.skhmus}
S_{\mu,h}^\pm(k)  =  (-1)^{s_{h,k}} \,2^{m} \\ \times \Big\{
\cos(\tfrac{2k\pi \mu}p) \Big(\prod_{j=1}^q \cos(\tfrac{jk\pi}p)
\Big)^{a} \mp i^{m+1} \sin(\tfrac{2k\pi \mu}p) \Big(\prod_{j=1}^q
\sin(\tf{jk\pi}p) \Big)^{a} \Big\}.
\end{multline}
Hence, by
\eqref{eq.multips} and \eqref{eq.skhmus}, we obtain
\begin{eqnarray*}
d_{\ell,\mu,h}^+ - d_{\ell,\mu,h}^- & = & \tfrac 1p \, \sum_{k=1}^{p-1} e^{\f{2\pi i
k\ell}{p}}  \big( S_{\mu,h}^+(k) - S_{\mu,h}^-(k) \big)
\\
&=&  -\tfrac{(2i)^{m+1}}{p} \, \sum_{k=1}^{p-1} (-1)^{s_{h,k}} \, e^{\f{2\pi
i k\ell}{p}} \, \sin(\tf{2k\pi \mu}p)\, \Big(\prod_{j=1}^q
\sin \big(\tf{jk\pi}p\big)\Big)^{a}.
\end{eqnarray*}
Now, by Lemma \ref{lem.prodtrigs}(i), \eqref{eq.signs}, and also using that $(-1)^\f{p^2-1}8 =
(-1)^{[\f{q+1}{2}]}$ and $aq=m$, we arrive at the desired expression.
\end{proof}

Our next goal is to find explicit expressions for $d^+_{\ell,\mu,h} - d^-_{\ell,\mu,h}$.
First, we fix some notations. Set $p=2q+1$ and $n=2m+1$. Then, since $b=0$ and $c=1$, we get
$n=a(p-1)+1 = 2aq+1$. Thus, $m=aq$ and we have
\begin{equation}\label{conditions2}
a \text{ even} \: \Rightarrow \: m=2r, \quad a \text{ odd}  \: \Rightarrow
\: \left\{ \begin{array}{ll}
p = 4t+1 \quad \Leftrightarrow  \quad m=2r, \sk \\
p = 4t+3 \quad \Leftrightarrow \quad  m=2r+1.
\end{array} \right.
\end{equation}

The proof of Lemma \ref{lem.Deltafinal} shows that  $\mu \in \N$ if $h=1$
and $\mu \in \N_0 + \tf 12$ if $h=2$.

\begin{proposition}\label{prop.Deltas}
Let $p=2q+1$ be a prime and let $\ell \in \N_0$ be chosen with $0\le \ell \le p-1$. Consider $D_\ell$ acting on an
exceptional
$\Z_p$-manifold of dimension $n = a(p-1)+1$ equipped with a spin structure $\vep_h$, $h=1,2$. Put $r=[\f n4]$.

\sk (i) If $a$ is even, then
$d_{0,\mu,1}^+ - d_{0,\mu,1}^- = d_{0,\mu,2}^+ - d_{0,\mu,2}^- =0$
and if $\ell\ne0$ we have
$$ d_{\ell,\mu,h}^+ - d_{\ell,\mu,h}^- = \left\{ \begin{array}{ll} \pm (-1)^{r} \, p^{\f a2}
& \qquad \textrm{if } p \,|\, h(\ell \mp \mu), \msk
\\ 0 & \qquad \textrm{otherwise.} \end{array}\right. $$

\sk (ii)
If  $a$ is odd, then
\begin{align*}
d_{\ell,\mu,h}^+ - d_{\ell,\mu,h}^- = (-1)^{q+r} \, \Big( \big(\tf{2(\ell - \mu) }{p} \big) - \big( \tf{2(\ell + \mu)}{p} \big) \Big) \, p^{\f{a-1}{2}}.
\end{align*}
In particular, for $\ell=0$ we have
$$ d_{0,\mu,h}^+ - d_{0,\mu,h}^-
= \left\{ \begin{array}{ll}
0 & \quad \text{if } p\equiv 1 \,(4), \sk \\
(-1)^{r} \, 2 \, \big(\tf{2\mu}{p} \big) \, p^{\f{a-1}{2}}  & \quad
\text{if } p\equiv 3 \,(4),
\end{array} \right.$$
where $\big( \f{\cdot}p \big)$ denotes the Legendre symbol and $\mu\in \tf 12\N_0$.
\end{proposition}
\begin{proof}
We define the integer
\begin{equation} \label{eq:cmu}
c_\mu = c(\mu,h) := \mu-\tf{\delta_{h,2}}{2} \in \N_0,
\end{equation}
where $\delta_{_{h,2}}$ is the Kronecker delta function. Note that the
expression in Lemma \ref{lem.Deltafinal} can be written as
\begin{equation*}\label{eq.Delta a}
d_{\ell,\mu,h}^+ - d_{\ell,\mu,h}^- \left\{ \begin{array}{ll}
-i^{m+1} \,2\,p^{\f{a}{2}-1} F_h^{\chi_{_0}}(\ell,c_\mu) & \qquad \text{if
$a$ even,} \msk \\
- i^{m+1} \,2\,  p^{\f{a}{2}-1} \,(-1)^{(\f{p^2-1}{8})} \,
F_h^{\chi_p}(\ell,c_\mu) & \qquad \text{if $a$ odd,}
\end{array} \right.
\end{equation*}
in the notations of Definition \ref{defi.sums}.

Assertion (i) follows directly from
Proposition \ref{prop.Qh} and from the previous expression.
Relative to assertion (ii), we can
apply Proposition \ref{prop.Fh} and the fact that $\big( \f 2p \big) =
(-1)^{\f{p^2-1}{8}}$ to get
\begin{align*}
d_{\ell,\mu,1}^+ - d_{\ell,\mu,1}^- & = i^m \, \delta(p)  \, p^{\f{a-1}{2}} \, \big( \tf
2p \big) \, \Big( \big( \tf{\ell-\mu}{p} \big) - \big(\tf{\ell+\mu}{p}\big) \Big), \\ d_{\ell,\mu,2}^+ - d_{\ell,\mu,2}^- & = i^m \, \delta(p)  \, p^{\f{a-1}{2}} \, \Big(
\big( \tf{2(\ell-c_\mu)-1}{p} \big) - \big(\tf{2(\ell+c_\mu)+1}{p}\big) \Big),
\end{align*}
where $\delta(p)=1$ if $p\equiv 1 \,(4)$ and $\delta(p)=i$ if $p\equiv 3 \,(4)$.
Note that by \eqref{conditions2} $i^m \delta(p) = (-1)^{q+r}$. The result follows from \eqref{eq:cmu} and the
multiplicativity of $(\tf .p)$. In particular, for $\ell=0$ we get the remaining assertion.
\end{proof}

\subsection{Eta series}\label{S3-ES}
We are now in a position to explicitly compute the twisted eta function $\eta_{\ell,h}(s)$ of a general spin $\Z_p$-manifold $(M,\vep_h)$.
We shall see that the expressions will be given in
terms of Hurwitz zeta functions
\begin{equation}\label{hurwitz}
\zeta(s,\alpha) = \sum_{n\ge 0} \tf{1}{(n+\alpha)^s}, \qquad
\mathrm{Re}(s)>1, \quad \alpha \in (0,1].
\end{equation}

\begin{theorem}\label{thm.eta series}
Let $p$ be an odd prime, $\ell\in \N_0$ with $0 \le \ell\le p-1$.
Let $(M,\vep_h)$,
be an exceptional spin $\Z_p$-manifold of dimension $n=a(p-1)+1$.
Put $r=[\tf n4]$ and $t=[\f{p}{4}]$. Then,
the eta series is given as follows:

\sk \emph{(i)} Let $a$ be even. Then $\eta_{0,h}(s) =0$, $h=1,2,$ and for
$1\le \ell \le p-1$ we have
\begin{align*}
& \eta_{\ell,1}
(s) = \tf{(-1)^r}{(2\pi p)^s} \, p^{\f a2} \, \big( \zeta(s,\tf \ell p) -
\zeta(s,\tf{p-\ell}{p}) \big), \\ \\
& \eta_{\ell,2}
(s) = \left\{ \begin{array}{ll} \tf{(-1)^r}{(2\pi p)^{s}} \,\ p^{\f a2} \,
\Big( \zeta(s,\tf 12 + \tf \ell p) - \zeta(s,\tf 12 - \tf \ell p) \Big) &
\quad  1 \le \ell \le
q, \msk \\
\tf{(-1)^{r}}{(2\pi p)^{s}} \, p^{\f a2} \, \Big( \zeta(s,\tf 12 -
\tf{p-\ell}{p}) - \zeta(s,\tf 12 + \tf{p-\ell}{p}) \Big) & \quad q< \ell <p.
\end{array} \right.
\end{align*}

\sk \emph{(ii)} Let $a$ be odd. Then, for $0\le \ell\le p-1$ we have
\begin{align*}
\eta_{\ell,1}(s) & = \tf{(-1)^{t+r}}{(2\pi p)^s} \, p^{\f{a-1}{2}} \, \sum_{j=1}^{p-1}
\Big( (\tf{\ell -j}{p}) - (\tf{\ell +
j}{p}) \Big) \, \zeta(s,\tf jp),
\sk \\
\eta_{\ell,2}(s) & = \tf{(-1)^{q+r}}{(\pi p)^s} \, p^{\f{a-1}{2}} \, \sum_{j=0}^{p-1}
\left( (\tf{2\ell-(2j+1)}{p}) -
(\tf{2\ell+(2j+1)}{p}) \right) \, \zeta(s,\tf{2j+1}{2p}).
\end{align*}
In particular, $\eta_{0,h}(s)=0$ for $p\equiv 1 \, (4)$.
\end{theorem}

\begin{proof}
By \eqref{eq.etaseriesdif} and \eqref{eq:cmu}, we have to compute the series
\begin{equation}\label{etaseriesdif2}
\eta_{\ell,h}(s) = \tf{1}{\pi^s} \, \sum_{c=1}^{\infty}
\f{d_{\ell,\mu,h}^+ - d_{\ell,\mu,h}^-}{(2c-\delta_{_{h,2}})^s}.
\end{equation}

We first prove (i). Let $a$ be even. By Proposition \ref{prop.Deltas} we have that $d_{0,\mu,h}^+ - d_{0,\mu,h}^- = 0$ and hence $\eta_{0,h}(s) = 0$, $h=1,2$. Also, for $1 \le \ell \le p-1$ we have $ d_{\ell,\mu,h}^+ - d_{\ell,\mu,h}^- = \pm (-1)^r \, p^{\tf a2}$
if $p \,|\, h(\ell \mp \mu)$ where $\mu = c_\mu + \f{\delta_{_{h,2}}}{2}$
with $c_\mu \in \N_0$; and $ d_{\ell,\mu,h}^+ - d_{\ell,\mu,h}^- = 0$ otherwise. Let $c=c_\mu
\ge 1$.

\msk (a) Take $h=1$ and $p\,|\, \ell \mp c$. Then $c = \mp (pk-\ell)$ for
some $k\in\Z$ and
\begin{eqnarray*}
c=\ell - pk \ge 1 \: \Leftrightarrow \: k\le 0, \qquad   c=pk - \ell \ge
1 \: \Leftrightarrow \: k\ge 1.
\end{eqnarray*}
Thus, by \eqref{etaseriesdif2} we get
\begin{eqnarray*}
\eta_{\ell,1}(s) &=&
\tf{(-1)^r}{(2\pi)^s}\, p^{\f a2} \, \Big( \sum_{k\le 0} \tf{1}{{(\ell -
pk)}^s} - \sum_{k\ge 1} \tf{1}{{(pk-\ell)}^s} \Big)\\& =&
\tf{(-1)^r}{(2\pi p)^s} \,  p^{\f a2} \, \big( \zeta(s,\tf \ell p) -
\zeta(s, \tf{p-\ell} p) \big).
\end{eqnarray*}

\msk (b) Take $h=2$ and $p\,|\, 2(\ell \mp c) \mp 1$. Then $2(\ell \mp c) \mp
1 = pk$, with $k$ odd. Thus, we have $2c+1=\pm
(2\ell-pk)$, $k$ odd, and
\begin{align*}
& 2\ell-pk \ge 1 \quad \Leftrightarrow  \quad \left\{ \begin{array}{ll}
k\le -1  &  \quad \text{if } 1\le \ell \le q, \sk \\ k\le 1 &
\quad \text{if } q<\ell <p, \end{array}\right. \msk \\  & pk-2\ell \ge 1
\quad \Leftrightarrow \quad
\left\{ \begin{array}{ll} k\ge 1 &  \quad \:\:\:\, \text{if }
1\le \ell \le q,  \sk \\ k\ge 3 & \quad \:\:\:\, \text{if } q<\ell <p.
\end{array}\right.
\end{align*}

Assume $1\le \ell \le q$. We have
\begin{eqnarray*}
\eta_{\ell,2}(s) &=& \tf{(-1)^r}{({2\pi p})^s}\, p^{\f a2} \, \Big(
\sum_{\begin{smallmatrix}
k\le -1 \\ k \text{ odd} \end{smallmatrix}} \f{1}{{(\frac \ell p -  \frac
k2)}^s} - \sum_{\begin{smallmatrix} k\ge 1 \\ k \text{ odd}
\end{smallmatrix}}
\f{1}{{( \frac k2- \frac \ell p)}^s} \Big) \\
&=& \tf{(-1)^r}{(2\pi p)^s} \, p^{\f a2} \, \Big( \sum_{n=0}^{\infty}
\f{1}{{(n+ \tf 12 +\tf \ell p)}^s} -
\sum_{n=0}^\infty \f{1}{{(n+ \tf 12- \tf \ell p)}^s} \Big) \\
&=& \tf{(-1)^r}{(2\pi p)^s} \, p^{\f a2} \, \big(
\zeta(s,\tf 12 + \tf \ell p) - \zeta(s,\tf 12 - \tf \ell{p}) \big),
\end{eqnarray*}
since $0<\f 12 \pm \f \ell{p} <1$ for $1\le \ell \le q$.

\sk
The case $q<\ell<p$ is a bit more involved. We have
\begin{eqnarray} \label{etaell2}
\eta_{\ell,2}(s) &=& \tf{(-1)^r}{(2 \pi p)^s}\, p^{\f a2} \, \Big(
\sum_{\begin{smallmatrix} k\le 1 \\ k \text{ odd}
\end{smallmatrix}} \f{1}{{( \tf \ell p - \tf k 2)}^s} -
\sum_{\begin{smallmatrix}
k\ge 3 \\ k \text{ odd} \end{smallmatrix}}
\f{1}{{(\tf k 2- \tf \ell p)}^s} \Big).
\end{eqnarray}
Now, the first sum in this expression equals
\begin{eqnarray*} \f{1}{{(\tf \ell p - \tf 12 )}^s} +
\sum_{\begin{smallmatrix} k\ge 1 \\ k \text{ odd} \end{smallmatrix}}
\f{1}{{(\tf k2 + \tf{\ell}p)}^s}  & = & \f{1}{{(\tf{\ell}{p}-\tf 12)}^s} +
\sum_{n = 0}^{\infty} \f{1}{(n+ \tf 12 + \f{\ell}{p})^s}
\\ & = &
\sum_{n = 0}^{\infty} \f{1}{(n + \tf 12 +
\f{\ell-p}{p})^s} =
\zeta(s,\tf 12 - \tf{p-\ell}{p})
\end{eqnarray*}
since $0<\f{p-\ell}{p} < \f 12$ for $q<\ell<p$. Similarly, the second sum
equals
\begin{eqnarray*}
\sum_{n\ge 1} \f{1}{{(n+\tf 12 -\tf{\ell}{p})}^s} =
\sum_{n \ge 0} \f{1}{{(n+ \tf 12+ \tf{p-\ell}{p})}^s} =
\zeta (s,\tf 12 +\tf{p-\ell}{p}),
\end{eqnarray*}
By substituting in \eqref{etaell2} we obtain the second expression for
$\eta_{\ell,2}(s)$ in (i).

\sk We now check  (ii). Let $a$ be odd and $0\le \ell \le p-1$. By using
\eqref{etaseriesdif2} and Proposition \ref{prop.Deltas} (ii), and
writing $c=pt+j$ with $t\ge 0$, $0 \le j \le p-1$, we get
\begin{eqnarray*}
\eta_{\ell,1}(s) & = & \tf{(-1)^{q+r}}{(2\pi)^s} (\tf 2p) \, p^{\f{a-1}2}
\sum_{c=1}^{\infty} \f{(\f{\ell-c}{p})-(\f{\ell+c}{p})}{c^s}
\\ &=&
\tf{(-1)^{t+r}}{(2\pi p)^s} p^{\f{a-1}2} \sum_{j=1}^{p-1} \big(
(\tf{\ell-j}{p})-(\tf{\ell+j}{p}) \big) \sum_{t=0}^\infty
\tf{1}{(t+ \tf jp)^s}
\end{eqnarray*}
where we have used that $(-1)^q (\tf 2p) = (-1)^{t}$. This gives the
expression of $\eta_{\ell,1}
(s)$. Similarly
\begin{eqnarray*}
\eta_{\ell,2}
(s) & = & \tf{(-1)^{q+r}}{2\pi^s} p^{\f{a-1}2} \sum_{c=0}^{\infty}
\f{(\f{2(\ell-c)-1}{p})-(\f{2(\ell+c)+1}{p})}{(c+\tf 12)^s} \\
&=& \tf{(-1)^{q+r}}{(2\pi p)^s} p^{\f{a-1}2} \sum_{j=0}^{p-1} \Big(
\big( \tf{2\ell-(2j+1)}{p} \big) - \big( \tf{2\ell+(2j+ 1)}{p} \big)
\Big) \sum_{t=0}^\infty \tf{1}{\big(t+\tf{2j+1}{2p}\big)^s}.
\end{eqnarray*}
Now, using that $\sum_{t=0}^\infty \big(t+\f{2j+1}{2p}\big)^{-s} =
\zeta(s,\tf{2j+1}{2p})$
in the previous equations, we obtain the expression in the statement. The
remaining assertion is clear and the theorem is thus proved.
\end{proof}

\begin{remark}
In the particular case when $\ell=0$, $b+c=1$ (i.e.\@ $\beta_1=1$), $a$ is odd and $p\equiv 3\,(4)$ (see Theorem 3.3 and Corollary 3.4), the untwisted eta series $\eta_{0,h} (s)$ were computed in \cite{MiPo09}. Some  easy calculations show that the expressions given there coincide with the corresponding ones in Theorem \ref{thm.eta series}.
\end{remark}

\section{Twisted eta invariants} \label{sect-4}
\subsection{Twisted and relative eta invariants}\label{S3-TERLI}
Here we compute the twisted eta invariants $\eta_\ell$ and $\bar \eta_\ell$, for any
$0\le \ell \le p-1$, the dimension of the kernel of $D_\ell$
and the twisted relative eta invariants, i.e.\@ the differences
$\bar \eta_\ell - \bar \eta_0$.

\sk
We will need the following notations. For $h=1,2$ we set
\begin{equation}\label{aux sums2}
S_h^{\pm}(\ell,p) := \sum_{j=1}^{p+\big[\f {h\ell}p \big]\,p-h\ell-1} \big( \tf jp \big)
\quad \pm \quad \sum_{j=1}^{h\ell-\big [\f {h\ell}p\big ]\,p-1} \big( \tf jp \big).
\end{equation}
where $\big(\frac{\cdot}{p}\big)$ stands for the Legendre symbol modulo $p$.
 Note that
\begin{equation}\label{s0p=0}
S_1^\pm(0,p) = S_1^\pm(0,p) = 0
\end{equation}
since $\sum_{j=1}^{p-1} \big( \tf jp \big) =0$.

\goodbreak
\begin{theorem}
\label{thm.etainvs}
Let $p=2q+1$ be an odd prime and let $\ell \in \N$ be such that $0\le \ell \le p-1$. Let $M$ be an exceptional
$\Z_p$-manifold of dimension $n=a(p-1)+1$. Put $r=[\tf n4]$ and $t=[\tf p4]$.
 The twisted eta invariants of
$(M,\vep_h)$ are given as follows.

\sk (i) If $a$ is even then $\eta_{0,h}(0) = 0$ and for $\ell \ne 0$ we have
$$\eta_{\ell,1} = (-1)^{r} p^{\f{a}{2}-1} (p -2\ell), \qquad
\eta_{\ell,2} = (-1)^{r} p^{\f{a}{2}-1}2([\tf {2\ell}p]\,p -\ell).$$

\sk (ii) If $a$ is odd then
\begin{align*}
\eta_{\ell,1} & = \left\{ \begin{array}{ll} (-1)^{t+r+1} p^{\f{a-1}{2}}
\, S_1^-(\ell,p)  & \qquad \quad p\equiv 1\,(4), \msk \\ (-1)^{t+r} p^{\f{a-1}{2}} \big( S_1^+(\ell,p) + \tf 2p
\sum\limits_{j=1}^{p-1} \big(\tf{j}{p}\big)j \big) & \qquad \quad p\equiv 3\,(4),
\end{array} \right.
\msk \\
\eta_{\ell,2} & = \left\{
\begin{array}{ll}
(-1)^{q+r+1} p^{\f{a-1}{2}} \big( S_2^-(\ell,p)   -   \big( \tf{2}{p} \big)
S_1^-(\ell,p) \big) & \quad \: p\equiv 1\,(4), \msk \\
(-1)^{q+r} p^{\f{a-1}{2}} \big\{ S_2^+(\ell,p)   +    \big( \tf{2}{p} \big)
\, S_1^+(\ell,p) \, + \\ \hfill + \, \big(1- (\tf 2p)
\big) \tf 2p \sum\limits_{j=1}^{p-1}  \big( \tf{j}{p} \big) j \big\} & \quad
\: p\equiv 3\,(4).
\end{array} \right.
\end{align*}
In particular, if $a$ is odd, we have that $\eta_{0,1} = 0$ for $p\equiv 1\,(4)$ and that $\eta_{0,2} = 0$ for both $p\equiv 1\, (4)$ or $p\equiv 7 \,(8)$.
\end{theorem}

\begin{proof}
We need only evaluate the expressions in Theorem \ref{thm.eta series} at
$s=0$, using that $\zeta(0,\alpha) = \tf{1}{2} - \alpha$.

\sk  (i) If $a$ is even, by Theorem \ref{thm.eta series} (i) we have
$$\eta_{\ell,1}(0) = (-1)^r p^{\f a2} \big[(\tf 12 - \tf{\ell}{p}) - (\tf 12
- \tf{p-\ell}{p})\big] = (-1)^r p^{\f a2} (1-\tf {2\ell}{p})$$
Proceeding similarly, we have
$$\eta_{\ell,2}(0) = \left\{ \begin{array}{ll}
(-1)^{r} \, p^{\f{a}{2} -1}\, (-2\ell) & \quad 1\le \ell \le q, \msk \\
(-1)^r p^{\f{a}{2}-1}\,  2(p-\ell) & \quad q< \ell <p,
\end{array} \right.$$
from where the expression in the statement follows.

\sk (ii) Assume now that $a$ is odd. By Theorem
\ref{thm.eta series} (ii), we have
\begin{align*}
\eta_{\ell,1}(0) & = (-1)^{t+r} \, p^{\f{a-1}{2}} \, \sum_{j=1}^{p-1} \Big( (\tf{\ell
-j}{p}) - (\tf{\ell + j}{p}) \Big) \,
(\tf 12 - \tf jp),
\msk \\
\eta_{\ell,2}
(0) & = (-1)^{q+r} \, p^{\f{a-1}{2}} \, \sum_{j=0}^{p-1} \left(
(\tf{2\ell-(2j+1)}{p}) - (\tf{2\ell+(2j+1)}{p}) \right) \, (\tf{p-1}{2p} -
\tf{j}{p}).
\end{align*}
Now, by applying Lemma \ref{lem.legendres}, in the notations of
\eqref{S1yS2}, we have
\begin{align*}
\eta_{\ell,1}(0) & = (-1)^{t+r+1} \, p^{\f{a-3}{2}} \, S_1(\ell,p), \msk \\
\eta_{\ell,2}(0) & = (-1)^{q+r+1} \, p^{\f{a-3}{2}} \, S_2(\ell,p).
\end{align*}
Finally, using Proposition \ref{prop.dif legendres} we get the desired expressions.

The remaining assertions follow from \eqref{s0p=0} and thus the theorem is now proved.
\end{proof}

We will now show the integrality of the  eta invariants $\eta_\ell$ (except for the 3-dimensional $\Z_3$-manifold $M_{3,1}$) and study their parity.
\sk
To this end, we first recall the
Dirichlet class number formula for a negative discriminant $D$ in the particular case $D=-p$, with $p \equiv 3 \,(4)$ a positive odd prime. It is given by
\begin{equation}\label{eq.dirichlet}
h_{-p} = -\tf{\omega_{-p}}{2p}\, \sum _{j=0}^{p-1}  \big( \tf{j}{p} \big) j
\in \Z,
\end{equation}
where $h_{-p}$ is the class number of the imaginary quadratic field
$\Q(\sqrt{-p})$ and $\omega_{-p}$ is the number of $p^{\mathrm{th}}$-roots
of unity in that field. Hence, $\omega_{-p}=6$ if $p=3$ and $\omega_{-p}=2$
if $p\ge 5$.

\begin{corollary}\label{parity}
Let $p$ be an odd prime and $\ell \in \N$ with $0\le \ell \le p-1$. Let $(M,\vep_h)$ be an exceptional spin $\Z_p$-manifold, $h=1,2$.

\sk (i) If $(p,a) \ne (3,1)$ then $\eta_{\ell,h} \in \Z$. Furthermore, $\eta_{0,h}$ is even and, if $\ell \ne 0$, then $\eta_{\ell,1}$ is odd and $\eta_{\ell,2}$ is even.

\sk (ii) If $(p,a) = (3,1)$ then
$$\eta_{\ell,1} =  \left\{ \begin{array}{rl} -2/3 & \quad \ell=0, \sk \\ 1/3 & \quad \ell=1,2,
\end{array}\right. \quad \text{ and } \quad \eta_{\ell,2} = 4/3 \quad \ell=0,1,2.$$
\end{corollary}

\begin{proof}
(i) Let $(p,a) \ne (3,1)$. If $a$ is even it is clear from the expressions in (i) of Theorem \ref{thm.etainvs} that $\eta_{\ell,h} \in \Z$ and $\eta_{0,h} \in 2\Z$. For $\ell \ne 0$, we also have that $\eta_{\ell,1}$ is odd and $\eta_{\ell,2}$ is even.

\sk
We now let $a$ be odd. We will first show that the values at 0 are integers. It is clear that the sums $S_1^\pm(\ell,p), S_2^\pm(\ell,p) \in \Z$. In the case $p\equiv 3\,(4)$ there is another term to consider. By \eqref{eq.dirichlet}, it follows that
$$ \tf{1 }{p}\, \sum _{j=0}^{p-1}  \big( \tf{j}{p} \big) j  = -\tf{2h_{-p} }{\omega_{-p}} = \left \{ \begin{array}{ll}-h_{-p}& \quad p\ge 5, \sk  \\ -2/3 &\quad p=3, \end{array} \right. $$
since $h_{-3}=1$. In this way, $\tf{1}{p}\, \sum _{j=0}^{p-1} \big(\tf{j}{p} \big) j \in \Z$ for $p\ge 5$, while for $p=3$ we have that
$p^{\f{a-1}{2}} \, \tf{1}{p}\, \sum _{j=0}^{p-1} \big( \tf{j}{p} \big) j = 3^{\f{a-1}{2}} \, \tf{(-2)}{3} \in \Z$
for $a>1$. In any case, we see that $\eta_{\ell,h} \in \Z$ for $(p,a)\ne(3,1)$.

\sk We now consider the parity of the sums $S_h^\pm(\ell,p)$, $h=1,2$. If $\ell=0$, all the sums are zero. If $\ell\ne 0$, $S_1^\pm(\ell,p) \equiv (p-\ell-1) + (\ell - 1) \equiv p \mod 2$, hence it is odd. Similarly we verify that $S_2^\pm(\ell,p)$ is odd, in this case.

Now, making use of these parity considerations and looking at the expressions in (i) of Theorem \ref{thm.etainvs} for  $a$ odd, we see that again $\eta_{0,h} \in 2\Z$ and $\eta_{\ell,1}$ is odd and $\eta_{\ell,2}$ is even, for $\ell\ne 0$.

\sk
(ii) Suppose $(p,a)=(3,1)$. We need evaluate the expressions in Theorem \ref{thm.etainvs} (ii) for $p\equiv 3\,(4)$. We have that $q=1$ and $r=t=0$. Using that $(\tf 13)=1$, $(\tf 23)=-1$ and $\sum_{j=1}^{2} (\tf j3)j = -1$ we have
\begin{equation*}\label{sl12}
\eta_{\ell,1} = S_1^+(\ell,3) - \tf 23, \qquad  
\eta_{\ell,2} = \tf 43 + S_1^+(\ell,3) - S_2^+(\ell,3).
\end{equation*}
Now, using \eqref{aux sums2} we have $S_1^+(0,3) = S_2^+(0,3) = 0$, by \eqref{s0p=0}, and it is easy to check that $S_1^+(\ell,3) = S_2^+(\ell,3) = 1$ for $\ell=1,2$. Substituting these values in the previous equations the
result follows.
\end{proof}

\begin{remark}\label{rem. hp}
By using the formula \eqref{eq.dirichlet}, the expressions for the  eta invariants $\eta_{\ell,h}$ in Theorem \eqref{thm.etainvs} can be put in terms of class numbers $h_{-p}$ when $a$ is odd and $p\equiv 3\,(4)$. In particular, by \eqref{s0p=0}, for an exceptional manifold in the untwisted case $\ell = 0$, we get the expressions
$$\eta_{0,1} =  -4 \, p^{\f{a-1}2} \, \tfrac{h_{-p}}{\omega_{-p}}, \qquad \eta_{0,2} = \big\{ \big( \tf{2}p \big) -1 \big\} \, \eta_{0,1},$$
where we have used that $r=[\f n4]$ and $t=[\f p4]$.
Thus, for $p=3$ we have $\eta_{0,1} = -2 \cdot 3^{\f{a-3}2}$ and $\eta_{0,2} = 4 \cdot 3^{\f{a-3}2}$. For $p\ge 7$ we may conclude that
$$\eta_{0,1} =  -2 \, p^{\f{a-1}2} \, h_{-p} \qquad \text{and} \qquad \eta_{0,2} =  \left\{ \begin{array}{ll} 0 & \qquad p\equiv 7 \; (8), \\  4 \, p^{\f{a-1}2}\, h_{-p} & \qquad p\equiv 3 \; (8). \end{array} \right. $$
These expressions coincide with the ones obtained in \cite{MiPo09}, Theorem 4.1.
\end{remark}

It is known that the dimension of the kernel of the Dirac operator $D_{\ell}$ coincides with
the number of independent harmonic spinors, which in turn equals the multiplicity of the eigenvalue 0. That is,
$$\dim \ker D_{\ell} = d_{\ell,0}.$$
We now compute this invariant for an arbitrary $\Z_p$-manifold.

\begin{proposition}\label{harmonics}
Let $p=2q+1$ be an odd prime. Let $M$ be a $\Z_p$-manifold with a spin structure $\vep_h$. Then, $d_{\ell,0,h}=0$ for any nontrivial spin structure $\vep_h$, $h\ne 1$, while for the trivial spin structure $\vep_1$ we have
\begin{equation}\label{eq.dimker}
d_{\ell,0,1} =  \tf{2^{\f{b+c-1}{2}}}{p} \Big( 2^{(a+b)q} +
(-1)^{(\f{p^2-1}{8}){(a+b)}} \, \big(p \delta_{\ell,0} - 1\big) \Big)
\end{equation}
where $0\le \ell \le p-1$ and $\delta_{\ell,0}$ is the Kronecker delta
function.

\sk In particular, if $b+c>1$ then $d_{\ell,0,1}$ is even for any $0\le \ell \le p-1$
while if $b+c=1$ then $d_{0,0,1}$ is even and $d_{\ell,0,1}$ is
odd for $\ell \ne 0$.
\end{proposition}

\begin{proof}
By \eqref{eq.harmonicspinors} we have $d_{\ell,0,h}=0$ for $h \ne 1$ and
\begin{eqnarray*}
d_{\ell,0,1} = \tf 1p \, \sum_{k=0}^{p-1} \, e^{\f{2\pi i k\ell}{p}}
\; {\chi}_{_{L_{n}}}(\vep_1(\g^k)).
\end{eqnarray*}
Using \eqref{eq.spinstructures} and the fact that $x(\theta)^k=x(k\theta)$,
$\theta\in \R$, $k\in \Z$, we have that
$\vep_1(\g^k) = (-1)^{k[\f{q+1}{2}](a+b)} \; x_{a+b}\big(\tf{k\pi}p,
\tf{2k\pi}p,\dots, \tf{qk\pi}p \big)$.
Now, applying \eqref{spinpmchars}, we get
$$d_{\ell,0,1} = \tf{2^m}{p} \, \sum_{k=0}^{p-1} \, (-1)^{k[\frac{q+1}{2}](a+b)}  \, \Big(\prod_{j=1}^q \cos
\big(\tf{jk\pi}p\big)\Big)^{a+b} \, e^{\f{2\pi i k\ell}{p}} \,.$$

By (ii) in Lemma \ref{lem.prodtrigs}, for $k>0$ we have
$$\Big(\prod_{j=1}^q \cos \big(\tf{jk\pi}p\big)\Big)^{a+b} =
\f{(-1)^{(k-1)(\f{p^2-1}{8})(a+b)}}{2^{(a+b)q}} \,.$$

Thus, we get
$$d_{\ell,0,1} = \tf{2^m}{p} \, \Big( 1 + \tf{(-1)^{(\frac{p^2-1}{8})(a+b)}}{2^{q(a+b)}} \, \sum_{k=1}^{p-1}
e^{\f{2\pi i k\ell}{p}} \Big) \,.$$
Expression \eqref{eq.dimker} now follows from the fact that $\sum_{k=1}^{p-1}  e^{\f{2\pi ik\ell}{p}}$ equals
 $p-1$ for $\ell=0$ and $-1$ for $1\le \ell \le p-1$.
Since $2m+1 = n = a(p-1) + bp + c$ and $p=2q+1$ we have that $b+c$ is odd and $m = (a+b)q +
(\tf{b+c-1}{2})$.

\sk
The remaining assertions are now clear from \eqref{eq.dimker} and the
proposition readily follows.
\end{proof}

\begin{remark} \label{remark}
By \eqref{etainvariant}, 
for a $\Z_p$-manifold we have 
\begin{equation} \label{eta res + har}
\bar \eta_{\ell,h} = \tf 12 (\eta_{\ell,h} + d_{\ell,0,h}) \,. 
\end{equation}
Using Theorem \ref{thm.etainvs} and Proposition \ref{harmonics} one could easily write down explicit expressions for the 
twisted eta invariants $\bar \eta_{\ell,h}$ and the relative eta invariants $\bar \eta_{\ell,h} - \bar \eta_{0}$ for $1\le \ell \le p-1$. These formulas are too complicated to write them down, because of
the sums $S_h^\pm(\ell,p)$ appearing in the expression for $\eta_{\ell,h}$. However, in the untwisted case they get explicit and closed expressions (see Corollary \ref{eta untwisted}). On the other hand, we are mainly interested in their values modulo $\Z$ (see Theorem \eqref{twisted etas mod Z}).
\end{remark}

\begin{corollary}  \label{eta untwisted}
Let $p$ be an odd prime. 
In the untwisted case, i.e.\@ $\ell=0$,
the 
eta invariants of an arbitrary 
$\Z_p$-manifold $(M,\vep_h)$ 
have the following expressions.

\sk (i) If $M$ is non-exceptional, i.e.\@ $(b,c) \ne (0,1)$, then
$$\bar \eta_{0,1} = \tf 1p \, 2^{\f{b+c-3}{2}} \big( 2^{(a+b)(\f{p-1}{2})} + (-1)^{(\f{p^2-1}{8}){(a+b)}} \, (p - 1) \big)$$  and $\bar \eta_{0,h} = 0$ for $h\ne 1$.

\sk (ii) If $M$ is exceptional, i.e.\@ $(b,c) = (0,1)$, then
$$\bar \eta_{0,1} = \left\{ \begin{array}{ll}
\tf{1}{2p} \big( 2^{\f{n-1}{2}} + (-1)^{(\f{p+1}{8})(n-1)} (p-1) \big) -2 \, p^{\f{a-1}{2}} 
\, \f{h_{-p}}{\omega_{-p}} &   a \text{ odd, } p \equiv 3\,(4), \\ \sk
\tf{1}{2p} \big( 2^{\f{n-1}{2}} + (-1)^{(\f{p+1}{8})(n-1)} (p-1) \big)
& \quad \text{otherwise,} \end{array} \right.$$
and
$$\bar \eta_{0,2} = \left\{ \begin{array}{ll}
\big( 1 - (\tf 2p) \big) \, 2 \ p^{\f{a-1}{2}} \, \f{h_{-p}}{\omega_{-p}} &  \qquad \qquad  a \text{ odd, } p \equiv 3\,(4), \\ \sk 0  & \qquad \qquad  \text{otherwise.} \end{array} \right.$$
\end{corollary}

\begin{proof}
The result follows directly by substituting the expressions obtained in Theorem \ref{thm.etainvs}, Proposition \ref{harmonics} and Remark \ref{rem. hp} in \eqref{eta res + har}, and considering the different cases involved.
\end{proof}

We are now in a position to prove Theorem \ref{twisted etas mod Z}, one of the main results in the paper.

\subsection*{Proof of Theorem 1.1}
We  need study the integrality (or not) of $\bar \eta_{\ell,h}$ in \eqref{eta res + har},  by looking at the parity of the numbers $\eta_{\ell,h}$ and $d_{\ell,0,h}$.

\sk  In the non-exceptional case, i.e.\@ $(b,c)\ne (0,1)$, by using Corollary \ref{parity} and Proposition \ref{harmonics} we have that $\eta_{\ell,h}=0$, and hence $\bar \eta_{\ell,h} = \tf 12 d_{\ell,0,h} \in \Z$ for any $0 \le \ell \le p-1$. In particular, $\bar \eta_{\ell,h}=0$ for $\ell\ne 0$.

\sk In the exceptional case, i.e.\@ $(b,c)=(0,1)$, we have the following results. If $(p,a)\ne (3,1)$ then
\begin{align*}
&  \bar \eta_{0,1} = \tf 12 (\mathrm{even} + \mathrm{even}) \in \Z,  &  \bar \eta_{0,2} = \tf 12 (\mathrm{even} + 0) \in \Z, \\
& \bar \eta_{\ell,1} = \tf 12 (\mathrm{odd} + \mathrm{odd}) \in \Z, & \bar \eta_{\ell,2} = \tf 12 (\mathrm{even }+ 0) \in \Z,
\end{align*}
for $\ell \ne 0$. Thus, we have $\bar \eta_{\ell,h} \equiv 0$ mod $\Z$ in this case.

\sk
If $(p,a)=(3,1)$ then
\begin{align*}
& \bar \eta_{0,1} = \tf 12 (\tf 23 +0) = -\tf13, &  \bar \eta_{0,2} = \tf 12 (\tf 43 + 0) = \tf 23, \\
& \bar \eta_{\ell,1} = \tf 12 (-\tf13 + 1) = \tf 23, & \bar \eta_{\ell,2} = \tf 12 (\tf 43 + 0) = \tf 23,
\end{align*}
and now we have $\bar \eta_{\ell,h}\equiv \frac 23$ mod $\Z$.

\sk The remaining assertion is now clear and the result  follows. \hfill $\square$

\section{Equivariant spin bordism}\label{sect-5}
This section is devoted to the proof of Theorem \ref{thm-1.2}. We first review the basic definitions and notions. Let $p$ be an odd prime. Let $M$ be a compact oriented smooth manifold of dimension $n$ without boundary. An  \textit{equivariant} {\it $\mathbb{Z}_p$-structure} $\sigma$ on $M$ is a principal $\mathbb{Z}_p$-bundle
$$\mathbb{Z}_p \rightarrow P \rightarrow M.$$
This structure can also be regarded as being either a representation of the fundamental group of each connected
component of $M$ to $\mathbb{Z}_p$, or as being the homotopy class of a smooth map from $M$ to the classifying space
$B\mathbb{Z}_p$. These are equivalent formulations and this explains the utility of the concept. The {\it trivial
$\mathbb{Z}_p$-structure} $\sigma_0$ is defined by taking the product principal bundle
$P=M\times\mathbb{Z}_p$ or, equivalently, by taking the trivial representation of the fundamental group, or else,
equivalently, 
by taking the constant map from $M$ to $B\mathbb{Z}_p$.

\sk
Let $(M_i,\pbgs_i)$ be compact oriented spin manifolds of dimension $n$. Let $M_1-M_2$ be the disjoint union of
$M_1$ and $M_2$ where we give $M_2$ the opposite orientation. One says that $M_1$ and $M_2$ are {\it
Spin-bordant} if there exists a compact spin manifold $N$ with boundary, so
that the boundary of $N$ is $M_1-M_2$ and so that the spin structure on $N$ restricts to induce the
given spin structures on the manifolds $M_i$. Spin-bordism induces an equivalence relation; let $[(M,\pbgs)]$
denote the associated equivalence class and let $\MSpin_n$ be the collection of equivalence classes. Disjoint union
and Cartesian product gives
$\MSpin_*$ the structure of a graded unital ring. We refer to
\cite{ABP66,ABP67,ABP69,Th54} for further details
concerning these and related structures.

Additionally suppose $\sigma_i$ are equivariant $\mathbb{Z}_p$-structures on the manifolds $M_i$. One says that
$(M_1,\sigma_1,\pbgs_1)$ is {\it $\mathbb{Z}_p$-equivariant Spin-bordant} to $(M_2,\sigma_2, \pbgs_2)$ if in addition the bounding manifold $N$ admits an equivariant $\mathbb{Z}_p$-structure which restricts to given structures on the manifolds $M_i$. Again, this is an equivalence relation and we let $\MSpin_n(B\mathbb{Z}_p)$ denote the associated equivariant spin bordism groups.

\sk
We wish to focus on the $\mathbb{Z}_p$-structure. Forgetting the
$\mathbb{Z}_p$-structure defines the {\it
forgetful map} from
$\MSpin_n(B\mathbb{Z}_p)$ to $\MSpin_n$ which splits by the inclusion which
associates to every spin manifold
the trivial $\mathbb{Z}_p$-structure $\sigma_0$. The \emph{reduced
equivariant bordism} groups
$\RMSpin_n(B\mathbb{Z}_p)$ are the kernel of the forgetful map, that is
$[(M,\sigma,\pbgs)]$
belongs to $\RMSpin_n(B\mathbb{Z}_p)$ if and only if $[(M,\pbgs)]=0$ in
$\MSpin_n$. These groups
play much the same role in studying equivariant bordism as the reduced
homology groups play in
the study of homology -- one has a natural isomorphism
$$\MSpin_n(B\mathbb{Z}_p) = \RMSpin_n(B\mathbb{Z}_p)\oplus \MSpin_n.$$
Cartesian product makes $\RMSpin_*(B\mathbb{Z}_p)$ into an $\MSpin_*$-module. We refer to Bahri \emph{et
al}.\@ \cite{BBDG, BBG, BG87, BG87a} for details concerning the additive
structure of these and other related
groups. The natural projection $\pi$ from
$\MSpin_n(B\mathbb{Z}_p)$ to $\RMSpin_n(B\mathbb{Z}_p)$ is the object of
study in Theorem
\ref{thm-1.2} and is defined by
$$\pi(M,\varepsilon,\sigma)=[(M,\varepsilon,\sigma)]-[(M,\varepsilon,\sigma_0)]\in
\RMSpin_n(B\mathbb{Z}_p)\,.$$
The following result follows from Lemma 3.4.2, Lemma 3.4.3, and Theorem 3.44 of
\cite{Gi-88}:
\begin{theorem}\label{thm-?}
Let $p$ be an odd prime.

\sk (i) If $n$ is even, then $\RMSpin_n(B\mathbb{Z}_p)=0$.

\sk (ii) If $n$ is odd, then $\RMSpin_n(B\mathbb{Z}_p)$ is a finite group and
all the torsion in
$\RMSpin_n(B\mathbb{Z}_p)$ is $p$-torsion. Furthermore,
$\RMSpin_n(B\mathbb{Z}_p)$ is generated as a $\MSpin_*$-module by the diagonal lens spaces $\mathbb{S}^{2k-1}/\mathbb{Z}_p$ for $2k-1\le n$.
\end{theorem}

The characteristic numbers of $\MSpin_*$ are the Pontrjagin numbers, the
Stiefel-Whitney numbers, and
connective $K$-theory numbers. By contrast, the characteristic numbers of
$\RMSpin_*$ are given by the twisted
eta invariant defined previously and these lie in $\mathbb{Q}/\mathbb{Z}$
and are torsion invariants. Let $D$ be the
Dirac operator and let $\tau$ be a representation of the spin group. We let $\eta_\ell^\tau$ be the
eta invariant of the Dirac operator with
coefficients in the bundle defined by the representation $\tau$ and twisted
by the character $\ell$. The following result follows from Lemma 3.4.2, Lemma 3.4.3, and Theorem 3.44 of
\cite{Gi-88} -- see also the discussion in  Lemma 4.7.3 and Lemma 4.7.4 of
\cite{Gi-95}. It motivated our investigation of the eta
invariant for flat $\mathbb{Z}_p$-manifolds in the
first instance:

\begin{theorem}\label{thm-2.3}
Let $p$ be an odd prime. Let $M$ be an oriented manifold of dimension $n$. Let $\varepsilon$ be a spin structure on $M$ and let $\sigma$ be an equivariant $\mathbb{Z}_p$-structure on $M$. Let $\mathcal{M}:=(M,\varepsilon,\sigma)$.

\sk (i) Let $1\le \ell\le p-1$ and let $\tau$ be a representation of the spin group. Then:

\sk \quad (a) $(\bar\eta_\ell^\tau - \bar\eta_0^\tau)(\mathcal{M})$ takes values in $\mathbb{Z}[\frac1p]/\mathbb{Z}$.

\sk \quad (b) If $\pi(\mathcal{M})=0$ in $\RMSpin_n(B\mathbb{Z}_p)$, then $(\bar \eta_\ell^\tau - \bar \eta_0^\tau)(\mathcal{M})=0$ in $\mathbb{R}/\mathbb{Z}$.

\sk (ii) If the twisted relative eta invariants $(\bar \eta_\ell^\tau - \bar \eta_0^\tau)(\mathcal{M})$ vanish for all
$\tau$ and $\ell$, then $\pi(\mathcal{M})$ vanishes in $\RMSpin_n(B\mathbb{Z}_p)$.
\end{theorem}

We can now prove one of the two main results in the paper.
\subsection*{Proof of Theorem \ref{thm-1.2}}
Let $(M,\vep)$ be a spin $\Z_p$-manifold. The canonical equivariant $\Z_p$-structure $\sigma_p$ is defined by the cover $$\Z_p \rightarrow T_\Lambda \rightarrow M,$$ where $T_\Lambda$ is the associated torus.
The trivial equivariant $\Z_p$-structure $\sigma_0$ is defined by the cover
$\Z_p \rightarrow M \times \Z_p \rightarrow M$.
The associated principal Spin bundle is flat and defined by an equivariant
$\mathbb{Z}_{2p}$-structure on $M$ which may or may not reduce to a $\mathbb{Z}_p$-structure. Let
$\tau$ be a representation of
$\mathrm{Spin}(n)$ and let $0\le\ell\le p-1$. Since the spin structure is flat and arises from a
representation of $\mathbb{Z}_{2p}$, the bundle defined by the representation
$\tau$ and twisted by the character $\ell$ is flat and is defined by some representation $\nu$ of
$\mathbb{Z}_{2p}$. We may decompose
$$\mathbb{Z}_{2p} = \mathbb{Z}_2 \oplus \mathbb{Z}_p.$$
Let $\vartheta$ be the non-trivial character of $\mathbb{Z}_2$. We may
then decompose the representation $\nu = \nu_1 \oplus \nu_2 \vartheta$
where $\nu_i \in \operatorname{Rep}(\mathbb{Z}_p)$. Expand the representations $\nu_1$ and $\nu_2$ in terms of the
characters $\rho_i$ in the form:
$$\nu_1=\sum_{0\le i\le p-1}n_i\rho_i\quad\text{and}\quad
\nu_2=\sum_{0\le i\le p-1}\tilde n_i\rho_i\,.$$
Here $n_i=n_i(\tau,\ell)$ and $\bar n_i=\bar n_i(\tau,\ell)$ are non-negative integers. Let
$\tilde{\vep}$ be the associated spin structure on $M$ arising from the $\mathbb{Z}_2$
twisting $\vartheta$.
Let $\mathcal{M} = (M, \vep, \sigma)$ and $\tilde{\mathcal{M}} = (M, \tilde \vep, \sigma)$.
The above discussion then yields:
\begin{eqnarray*}
&&(\bar \eta_\ell^\tau - \bar \eta_0^\tau)(\mathcal{M}) = \sum_{i=1}^{p-1} \big\{ n_i(\bar \eta_i - \bar \eta_0)(\mathcal{M}) + \tilde n_i(\bar \eta_i - \bar \eta_0)(\tilde{\mathcal{M}}) \big\}.
\end{eqnarray*}
Theorem \ref{twisted etas mod Z} shows that $(\bar \eta_\ell^\tau - \bar \eta_0^\tau)(\mathcal{M})$
vanishes in $\mathbb{Q}/\mathbb{Z}$. Theorem \ref{thm-1.2} now follows by Theorem \ref{thm-2.3} (ii)
since $\tau$ and $\ell$ were arbitrary.
\hfill\qedbox

\begin{remark} \label{peters-remark}
Here we show how the reduced equivariant spin bordism groups come up in
some  questions in geometry.
Let $M$ be a connected spin manifold with finite fundamental group $\pi$
which admits a metric
of positive scalar curvature. The formula of Lichnerowicz \cite{L63} shows
that the kernel of the Dirac operator is
necessarily trivial. From this it follows that the index of the spin
operator vanishes and hence the generalized
$\hat A$-genus vanishes -- the generalized $\hat A$-genus is a topological
invariant which can be computed purely
combinatorially; it takes values either in $\mathbb{Z}$ or in $\mathbb{Z}_2$
depending upon the underlying dimension
of the manifold. Stolz \cite{S92} used the absolute spin bordism groups
$\MSpin_*$ to show that the generalized $\hat A$-genus was the only
obstruction to $M$ admitting a metric of positive scalar curvature if the
fundamental group $\pi$ was trivial. If $\pi=\mathbb{Z}_p$ or, more
generally, if $\pi$ is a spherical space form group, then one can define an
equivariant $\hat A$-genus and establish similar topological necessary and
sufficient conditions for $M$ to admit a metric of positive scalar curvature
\cite{BGS97}. We refer to \cite{GLP99} for further details about this area;
the reduced equivariant spin bordism groups $\RMSpin_n(B\mathbb{Z}_p)$ and
the associated eta invariants play a central role in the discussion.
\end{remark}

\section{Appendix: additional computations}\label{appendix}
Here we gather all the extra computations that were needed to obtain the
results in Section \ref{sect-3}. We compute some trigonometric
products of special values used to determine the asymmetric contribution of
the eigenvalues to the eta series, some twisted Gauss sums
appearing in the eta series and several sums involving Legendre symbols
appearing in the computation of the eta-invariants.

We recall here that the Legendre symbol $\big(\f{\cdot}{p} \big)$ is
$p$-periodic and satisfies
\begin{equation}\label{legendre2p}
(\tf2p) = (-1)^{\f{p^2-1}{8}}, \qquad (\tf{-1}{p}) = (-1)^{\f{p-1}{2}}.
\end{equation}

\subsection{Some trigonometric products}\label{sect-a-STP}
\begin{lemma}\label{lem.prodtrigs} Let $p=2q+1$ be an odd prime and let $k
\in \N$. Then we have
\begin{align*}\label{prod sines}
(i) \qquad \prod_{j=1}^q \sin(\tf{jk\pi}p) & = \left\{
\begin{array}{ll}
(-1)^{(k-1)(\f{p^2-1}{8})}\left( \f kp\right ) \, 2^{-q} \, \sqrt p & \quad
\text{ if }\; (k,p)=1, \msk \\
0 & \quad \text{ if } \; (k,p)>1,
\end{array} \right.
 \\ \\
(ii) \qquad \prod_{j=1}^q \cos(\tf{jk\pi}p) & = \left\{
\begin{array}{ll} (-1)^{(k-1)(\f{p^2-1}{8})} 2^{-q} & \qquad
\qquad \:\:\, \text{ if }\; (k,p)=1, \msk \\  (-1)^{\f kp [\f
{q+1}2]} & \qquad \qquad \:\:\; \text{ if } \; (k,p)>1.
\end{array} \right.
\end{align*}
\end{lemma}

\begin{proof}
Formula (i) in Lemma \ref{lem.prodtrigs} is proved in \cite{MiPo09}, Lemma
3.2. We check the second expression in the case $(k,p)=1$.
By (i) in the Lemma, we have
\begin{eqnarray*}
\prod_{j=1}^q \cos(\tf{jk \pi}{p}) & = &
\f{\prod\limits_{j=1}^q \sin(\f{2jk \pi}{p})}{2^q \prod\limits_{j=1}^q
\sin(\f{j k \pi}{p})} =
\f{(-1)^{(2k-1)(\f{p^2-1}{8})} \left( \f{2k}{p} \right
)}{(-1)^{(k-1)(\f{p^2-1}{8})}\left(\f kp \right)2^q}.
\end{eqnarray*}
Canceling terms and using \eqref{legendre2p} the assertion in the proposition follows.
\end{proof}

\subsection{Twisted character Gauss sums}\label
{Sect-A-TCGS}
Here we will compute the values of certain twisted character Gauss sums.

\sk We recall the character Gauss sum associated to the quadratic Dirichlet
character given by the Legendre symbol $(\f{\cdot}{p})$
modulo $p$ for $l\in \N$,
\begin{equation}\label{eq:Gausssum}
G(l,p) =  \sum_{k=0}^{p-1} \left(\tf kp \right) e^{\tf{2\pi i l k}{p}}
\end{equation}
and the special values
\begin{equation}\label{eq.G(1,p)}
G(1,p) = \delta(p) \sqrt{p}, \qquad G(l,p) = G(1,p)\left(\tf lp \right)
\end{equation}
where we have put
\begin{equation}\label{deltap} \delta(p) :=  \left\{ \begin{array}{ll} 1 &
\qquad p\equiv 1 \;(4), \sk\\ i   &
\qquad p\equiv 3 \;(4).
\end{array} \right.
\end{equation}

\begin{definition}\label{defi.sums}
Let $p=2q+1$ be an odd prime, $l,c\in \N$ and $h=1,2$. For $\chi$ a
character mod $p$ define the sums
\begin{align}
\label{eq.Gh} & G_h^\chi(l)   := \sum_{k=1}^{p-1} (-1)^{k(h+1)} \, \chi(k)
\:
e^{\tf{\pi i k \,(2l + \delta_{_{h,2}})}{p}}, \sk \\
\label{eq.Fh} & F_{h}^\chi(l,c) := \sum_{k=1}^{p-1} (-1)^{k(h+1)} \, \chi(k)
\; e^{\tf{2 \pi i l k}{p}} \, \sin \big(\tf{ \pi k\,(2c +
\delta_{_{h,2}})}{p}\big).
\end{align}
\end{definition}

We are interested in these sums only for $\chi = \chi_0$, the trivial
character $\mod p$, and for $\chi =
(\f{\cdot}{p})$, the quadratic character mod $p$ given by the Legendre
symbol. We will denote these characters by
$\chi_{_0}$ and $\chi_p$, respectively.

Thus, for example, $G_1^{\chi_p}(l) = G(l,p)$ is the standard character
Gauss sum in \eqref{eq:Gausssum} and
$G_2^{\chi_p}(l,p)$ corresponds to the shifted alternating Gauss sum
\begin{equation}\label{H(l,p)}
 G_2^{\chi_p}(l) = \sum_{k=0}^{p-1} (-1)^k \left(\tf kp \right)
e^{\tf{(2l+1)\pi ik}{p}} = \tf{p}{G(1,p)} \left(\tf {q-l}p \right).
\end{equation}
(See \cite{MiPo09}, Theorem 5.1. \textit{Note:} there $G_2^{\chi_p}(l,p)$
was denoted by $\tilde{H}(l,p)$. We
note that a factor $p$ is missing in that expression, although not in their
proof.)

We will also make use of the identity
\begin{equation}\label{leg q-l}
\big( \tf{l-q}{p}\big) = \big( \tf{2}{p}\big)\big( \tf{2l-2q}{p}\big) =
\big( \tf{2}{p}\big)\big( \tf{2l+1}{p}\big).
\end{equation}

\subsubsection{\textbf{Computation of the sums $G_h^\chi(l)$}} We now find
the values of $G_h^\chi(l)$ for
$\chi = \chi_{_0},  \chi_p$. These sums are modifications of sums of
$p^{\mathrm{th}}$-roots of unity and of Gauss sums.
\begin{proposition}\label{prop.Rh}
Let $p$ be an odd prime and $l\in \N$. Then,
\begin{align*}
G_1^{\chi_{_0}}(l) & = \left\{ \begin{array}{ll} p-1 & \quad p \,|\, l,
\msk \\  -1 & \quad p \nmid l, \end{array}\right. \quad
\text{and} \quad G_2^{\chi_{_0}}(l) = \left\{ \begin{array}{ll} p-1 & \quad
p \,|\, 2l + 1, \msk \\  -1 & \quad p \nmid 2l + 1.
\end{array} \right.
\end{align*}
In particular, $G_h^{\chi_{_0}}(l) \in \Z$, $h=1,2$.
\end{proposition}

\begin{proof}
Since $G_h^{\chi_{_0}}(l)$ is $p$-periodic we may assume that $0\le l \le
p-1$.

By \eqref{eq.Gh} we have $G_1^{\chi_{_0}}(l) = \sum_{k=1}^{p-1} e^{\f{2l \pi
i k}{p}}$. Clearly, $G_1^{\chi_{_0}}(0) = p-1$ and if
$1\le l \le p-1$, then $1 + G_1^{\chi_{0}}(l) = \sum_{k=0}^{p-1} e^{\f{2l
\pi i k}{p}} = 0$, and hence $G_1^{\chi_{_0}}(l) = -1$.

Now, $G_2^{\chi_{_0}}(l) = \sum_{k=1}^{p-1} (-1)^{k}e^{\f{(2l + 1) \pi i
k}{p}}$, by \eqref{eq.Gh}. If $p\,|\,2l+1$
then $2l+1 = p\alpha$ with $\alpha$ odd. Thus, $G_2^{\chi_{_0}}(l) =
\sum_{k=1}^{p-1} (-1)^k(-1)^k = p-1$. If $p\nmid
2l+1$ then, denoting by $\omega_l = e^{\f{(2l+1)\pi i k}{p}}$ and using
geometric summation, we have
$$1 + G_2^{\chi_{_0}}(l) = \sum_{k=0}^{p-1} (-1)^k \, \omega_l^k =
\frac{\omega_l^p+1}{\omega_l+1} = 0,$$
since $\omega_l^p = -1$ and $\omega_l\ne 1$. Thus, $G_2^{\chi_{_0}}(l) = -1$
in this case.
\end{proof}

\begin{proposition}\label{prop.Gh}
Let $p$ be an odd prime and $l\in \N$. Then,
\begin{equation*}
G_1^{\chi_p}(l) = \delta(p) \left(\tf l p \right) \sqrt{p} \qquad \text{and}
\qquad G_2^{\chi_p}(l) = \delta(p) \left(\tf 2p \right)
\left(\tf{2l+1}{p} \right) \sqrt{p},
\end{equation*}
where $\delta(p)$ is as defined in \eqref{deltap}. In particular,
$G_1^{\chi_p}(l)=0$ if $p\,|\,l$ and
$G_2^{\chi_p}(l)=0$ if $p\,|\,2l+1$.
\end{proposition}

\begin{proof}
If $h=1$, we have  $G_1^{\chi_p}(l) = G(l,p) = \delta(p) \big(\tf l p \big)
\sqrt{p}$ by \eqref{eq.Gh} and  \eqref{eq.G(1,p)}. If
$h=2$, by \eqref{H(l,p)} and \eqref{leg q-l}, we have
$$G_2^{\chi_p}(l) = \tf{p}{G(1,p)} \left(\tf{q-l}p \right) =
\tf{1}{\delta(p)} \big(\tf{-2}{p} \big) (\tf{2l+1}{p}) \sqrt p = \delta(p)
\big(\tf{2}{p} \big) (\tf{2l+1}{p}) \sqrt
p$$ since $\f{1}{\delta(p)} \big(\tf{-1}{p} \big) = \delta(p)$, and the
result follows.
\end{proof}

\subsubsection{\textbf{Computation of the sums $F_h^\chi(l,c)$}} We now find
the values of
$F_h^{\chi}(l,c)$ for $\chi = \chi_{_0}, \chi_p$.

\begin{proposition}\label{prop.Qh}
Let $p$ be an odd prime, $l\in \N_0$ and $c \in \N$. If $p\,|\,l$ then
$F_h^{\chi_{_0}}(l,c)=0$. If $p \nmid l $, then
\begin{equation*}
F_1^{\chi_{_0}}(l,c) = \left\{ \begin{array}{cl} \pm  \f{ip}{2}, & \:
\text{if } p\,|\,l \mp c, \msk \\
0 & \: \text{otherwise}, \end{array} \right. \quad
F_2^{\chi_{_0}}(l,c) = \left\{ \begin{array}{cl} \pm \f{ip}{2} & \: \text{if
} p\,|\,2(l \mp c) \mp 1, \msk \\
0 & \: \text{otherwise}.
\end{array}\right.
\end{equation*}
\end{proposition}

\begin{proof}
By \eqref{eq.Fh}, we have
\begin{eqnarray*}
&&F_1^{\chi_{_0}}(l,c) = \sum_{k=1}^{p-1}  e^{\tf{2 \pi i k l}{p}} \, \sin
\big(\tf{2c \pi k}{p}\big), \\
&&F_2^{\chi_{_0}}(l,c) = \sum_{k=1}^{p-1}  (-1)^k \, e^{\tf{2 \pi i k
l}{p}} \, \sin \big(\tf{(2c+1) \pi k}{p}\big)\,.
\end{eqnarray*}

If $p\,|\,l$, then $e^{\f{2 \pi i k l}{p}}=1$ and hence we have
$F_1^{\chi_{_0}}(l,c) = \mathrm{Im} \, G_1^{\chi_{_0}}(c) = 0$ and
$F_2^{\chi_{_0}}(l,c)  = \mathrm{Im} \, G_2^{\chi_{_0}}(c) = 0$.

\sk Now, if $p\nmid l$, then using trigonometric identities
\ref{prop.Rh}
we have that the real and imaginary parts of $F_1^{\chi_{_0}}(l,c)$ are
respectively given by
\begin{eqnarray*}
\sum_{k=1}^{p-1}  \cos \big( \tf{2 \pi k l}{p} \big) \, \sin \big(\tf{2c \pi
k}{p} \big) &=& \tf 12 \sum_{k=1}^{p-1} \sin \big(\tf{2\pi
k (l + c)}{p}\big) - \sin \big(\tf{2 \pi k(l-c)}{p}\big) \\ &=& \tf12 \big\{
\mathrm{Im} \, G_1^{\chi_{_0}}(l+c) - \mathrm{Im} \,
G_1^{\chi_{_0}}(l-c) \big\},
\end{eqnarray*}
\begin{eqnarray*}
\sum_{k=1}^{p-1}  \sin \big( \tf{2 \pi k l}{p} \big) \, \sin \big( \tf{2c
\pi k}{p} \big) & = & \tf 12 \sum_{k=1}^{p-1} \cos
\big(\tf{2\pi k (l-c)}{p}\big) - \cos \big(\tf{2 \pi k (l+c)}{p}\big) \\
&=& \tf12 \big\{ \mathrm{Re} \, G_1^{\chi_{_0}}(l-c) - \mathrm{Re} \,
G_1^{\chi_{_0}}(l+c) \big\}.
\end{eqnarray*}
Thus, by Proposition \ref{prop.Rh} we get
$$\mathrm{Re} \, F_1^{\chi_{_0}}(l,c) =0, \qquad \mathrm{Im} \,
F_1^{\chi_{_0}}(l,c) = \left\{ \begin{array}{cc} \pm \f{p}2 & \quad
p\,|\,l \mp c,  \sk \\ 0 & \quad \text{otherwise},
\end{array}\right.
$$
Similarly, we have
\begin{eqnarray*}
\mathrm{Re} \, F_2^{\chi_{_0}}(l,c) & = & \sum_{k=1}^{p-1} (-1)^k \, \cos
\big( \tf{2 \pi k l}{p} \big) \, \sin
\big( \tf{(2c+1) \pi k}{p} \big)
\\ &=& \tf 12 \sum_{k=1}^{p-1} (-1)^k \, \big\{ \sin \big(\tf{(2(l+c)+1)\pi
k}{p}\big) - \sin \big(\tf{(2(l-c)-1) \pi k)}{p}\big) \big\} \\
&=& \tf12 \big\{ \mathrm{Im} \, G_2^{\chi_{_0}}(l+c) - \mathrm{Im} \,
G_2^{\chi_{_0}}(l-c-1) \big\},
\end{eqnarray*}
\begin{eqnarray*}
\mathrm{Im} \, F_2^{\chi_{_0}}(l,c) & = & \sum_{k=1}^{p-1}  (-1)^k \, \sin
\big( \tf{2 \pi k l}{p} \big) \,
\sin \big( \tf{(2c+1) \pi k}{p} \big) \\
&=& \tf 12 \sum_{k=1}^{p-1} (-1)^k \big\{ \cos \big(\tf{(2(l-c)-1)\pi k
}{p}\big) - \cos \big(\tf{(2(l+c)+1) \pi k}{p}\big) \big\} \\
&=& \tf12 \big\{ \mathrm{Re} \, G_2^{\chi_{_0}}(l-c-1) - \mathrm{Re} \,
G_2^{\chi_{_0}}(l+c) \big\}
\end{eqnarray*}
and hence by Proposition \ref{prop.Rh}
$$\mathrm{Re} \, F_2^{\chi_{_0}}(l,c) = 0, \qquad \mathrm{Im} \,
F_2^{\chi_{_0}}(l,c)
= \left\{ \begin{array}{cc} \pm \f{p}2 & \: p\,|\,2(l \mp c) \mp 1,  \sk \\ 0 & \: \text{otherwise}.
\end{array}\right.$$

We thus get the expressions in the statement.
\end{proof}

\begin{proposition}\label{prop.Fh} Let $p$ be an odd prime and $l,c \in \N$.
Thus, we have
$$ F_h^{\chi_p}(l,c) = \left\{ \begin{array}{ll}
 \, i \, \delta(p)   \Big( \big( \tf{l-c}{p} \big) - \big( \tf{l+c}{p} \big)
\Big) \,\f{\sqrt p}{2}  & \qquad h=1, \msk \\
i\, \delta(p)  \,  \Big(\f{2}{p}\Big) \Big( \big( \tf{2(l-c)-1}{p} \big) -
\big( \tf{2(l+c)+1}{p} \big)
\Big)\,\f{\sqrt p}{2} & \qquad h=2,
\end{array} \right.$$
where $\delta(p)$ is defined in \eqref{deltap}.

\sk In particular, if $p\,|\,l$ then
\begin{align*}
F_1^{\chi_p}(l,c) & = \left\{ \begin{array}{ll} 0 & \qquad  \qquad p \equiv
1 \,(4),  \sk \\
\big( \f{ c}{p}  \, \big) \sqrt p & \qquad  \qquad p \equiv 3\,(4),
\end{array} \right. \msk \\
F_2^{\chi_p}(l,c) & = \left\{ \begin{array}{ll} 0 & \quad  \,      p \equiv
1 \,(4), \sk \\
\big( \f 2p \big) \big( \f{2c+1}{p} \big) \, \sqrt p   & \quad  \, p \equiv
3\,(4).
\end{array} \right.
\end{align*}
\end{proposition}

\begin{proof}
By \eqref{eq.Fh}, we have
\begin{eqnarray*}
F_1^{\chi_p}(l,c) & = & \sum_{k=1}^{p-1}  \big(\tf kp\big) \, e^{\tf{2 \pi i
k l}{p}} \, \sin \big(\tf{2c \pi k}{p}\big),
\\ F_2^{\chi_p}(l,c) & = & \sum_{k=1}^{p-1}  (-1)^k \, \big(\tf kp \big) \,
e^{\tf{2 \pi i k l}{p}} \, \sin \big(\tf{(2c+1) \pi k}{p}\big).
\end{eqnarray*}

If $h=1$, using trigonometric identities, the real and imaginary parts of
$F_1^{\chi_p}(l,c)$ are respectively given by
\begin{eqnarray*}
\sum_{k=1}^{p-1} \big(\tf kp\big) \, \cos \big( \tf{2 \pi k l}{p} \big) \,
\sin \big(\tf{2c \pi k}{p} \big) &=& \tf 12 \sum_{k=1}^{p-1}
\big(\tf kp\big) \, \big\{ \sin \big(\tf{2\pi k (l+c)}{p}\big) - \sin
\big(\tf{2 \pi k (l-c)}{p}\big) \big\}
\\ &=& \tf12 \big\{ \mathrm{Im} \, G_1^{\chi_p}(l+c) - \mathrm{Im} \,
G_1^{\chi_p}(l-c) \big\},
\end{eqnarray*}
\begin{eqnarray*}
\sum_{k=1}^{p-1} \big( \tf{k}{p} \big) \, \sin \big( \tf{2 \pi k l}{p} \big)
\, \sin \big( \tf{2c \pi k}{p} \big) &=& \tf 12
\sum_{k=1}^{p-1} \big(\tf kp\big) \, \big\{ \cos \big( \tf{2\pi k (l-c)}{p}
\big) - \cos \big( \tf{2 \pi k (l+c)}{p} \big) \big\}
\\ &=& \tf12 \big\{ \mathrm{Re} \, G_1^{\chi_p}(l-c) - \mathrm{Re} \,
G_1^{\chi_p}(l+c) \big\}
\end{eqnarray*}
Thus, by Proposition \ref{prop.Gh} we get
\begin{equation*}
\begin{split}
\mathrm{Re} \, F_1^{\chi_p}(l,c) & = \left\{ \begin{array}{ll} 0 & \qquad p
\equiv 1 \,(4), \sk \\ \, \Big( \big( \f{l+c}{p} \big) -
\big( \f{l-c}{p} \big)  \Big) \f{\sqrt p}{2} & \qquad p \equiv 3\,(4),
\end{array}\right. \msk \\
\mathrm{Im} \, F_1^{\chi_p}(l,c) & = \left\{ \begin{array}{ll} \, \Big(
\big( \f{l-c}{p} \big) - \big( \f{l+c}{p} \big)
\Big) \f{\sqrt p}{2}& \qquad p \equiv 1 \,(4), \sk \\ 0 & \qquad p \equiv 3
\,(4), \end{array} \right.
\end{split}
\end{equation*}
and hence
\begin{equation}\label{enero}
F_1^{\chi_p}(l,c) = i\delta(p) \Big( \big( \tf{l+c}{p} \big) - \big(
\tf{l-c}{p} \big)  \Big) \tf{\sqrt p}{2}.
\end{equation}

\sk Similarly, for $h=2$, we have
\begin{eqnarray*}
\mathrm{Re} \, F_2^{\chi_p}(l,c) & = & \sum_{k=1}^{p-1} (-1)^k \, \big(
\tf{k}{p} \big) \,  \cos \big( \tf{2 \pi k l}{p} \big) \, \sin
\big( \tf{(2c+1) \pi k}{p} \big) \\
&=& \tf 12 \sum_{k=1}^{p-1} (-1)^k \, \big( \tf{k}{p} \big) \, \big\{\sin
\big(\tf{(2(l+c)+1)\pi
k}{p}\big) - \sin \big(\tf{(2(l-c-1)+1) \pi k)}{p}\big) \big\} \\
&=& \tf12 \big\{ \mathrm{Im} \, G_2^{\chi_p}(l+c) - \mathrm{Im} \,
G_2^{\chi_p}(l-c-1) \big\}
\end{eqnarray*}
\begin{eqnarray*}
\mathrm{Im} \, F_2^{\chi_p}(l,c) & = & \sum_{k=1}^{p-1}  (-1)^k \, \big(
\tf{k}{p} \big) \, \sin \big( \tf{2 \pi k l}{p} \big) \, \sin
\big( \tf{(2c+1) \pi k}{p} \big) \\ &=& \tf 12 \sum_{k=1}^{p-1} (-1)^k \,
\big( \tf{k}{p} \big) \big\{ \cos \big(\tf{(2(l-c-1)+1)\pi k
}{p}\big) - \cos \big(\tf{(2(l+c)+1) \pi k}{p}\big) \big\} \\
&=& \tf12 \big\{ \mathrm{Re} \, G_2^{\chi_p}(l-c-1) - \mathrm{Re} \,
G_2^{\chi_p}(l+c) \big\}
\end{eqnarray*}
Again, by Proposition \ref{prop.Gh} we get
\begin{equation*}
\begin{split}
\mathrm{Re} \, F_2^{\chi_p}(l,c) & = \left\{ \begin{array}{ll} 0 & \quad p
\equiv 1 \,(4), \sk \\ \big(\f 2p \big) \Big(
\big(\f{2(l+c)+1}{p} \big) - \big(\f{2(l-c)-1}{p}\big) \Big) \f{\sqrt p }{2}
& \quad p \equiv 3 \,(4), \end{array} \right. \\
\mathrm{Im} \, F_2^{\chi_p}(l,c) & = \left\{ \begin{array}{ll}   \big(\f 2p
\big) \Big( \big(\f{2(l-c)-1}{p}\big) -
\big(\f{2(l+c)+1}{p}\big) \Big) \f{\sqrt p}{2} & \quad p \equiv 1 \,(4), \sk
\\ 0 & \quad p \equiv 3 \,(4),
\end{array} \right.
\end{split}
\end{equation*}
and hence
\begin{equation}\label{febrero}
F_2^{\chi_p}(l,c) =  i\delta(p) \big(\tf 2p \big) \Big(
\big(\tf{2(l-c)-1}{p} \big) - \big(\tf{2(l+c)+1}{p}\big) \Big)
\tf{\sqrt p }{2}.
\end{equation}

By \eqref{enero} and \eqref{febrero} we get the first formula in the
statement. The remaining assertion is easy
to check, and the proposition follows.
\end{proof}

\subsection{Sums involving Legendre symbols}\label{Sect-a-SILS}
Here we compute some sums involving Legendre symbols that were used in the
body of the paper. We will use the fact that
$\sum_{j=1}^{p-1}\big(\f jp\big) =0$ for any prime $p$.

\begin{lemma}\label{lem.legendres}
Let $p$ be an odd prime and $\ell \in \N$ with $0\le \ell \le p-1$. Then,
\begin{align}\label{klj}
& \sum_{j=1}^{p-1} \big( \tf{k \ell \pm j}{p} \big ) = - \big( \tf{k
\ell}{p} \big), \qquad k\in \Z,\\ \label{2l-2j+1} &
\sum_{j=0}^{p-1} \big( \tf{2 \ell \pm (2j+1)}{p} \big) = 0.
\end{align}
\end{lemma}
\begin{proof}
First, note that
$$\sum_{j=1}^{p-1} \big( \tf{\ell+j}{p} \big) = \sum_{\begin{smallmatrix}
j=1 \\ j\ne \ell \end{smallmatrix}}^{p-1} \big( \tf{j}{p} \big)
= - \big( \tf{\ell}{p} \big),$$ and hence also,
$$\sum_{j=1}^{p-1} \big( \tf{\ell-j}{p} \big) = \sum_{j=1}^{p-1} \big(
\tf{-(j-\ell)}{p} \big) =
\big( \tf{-1}{p} \big) \sum_{j=1}^{p-1} \big( \tf{j+(p-\ell)}{p} \big) = -
\big( \tf{-1}{p} \big) \big( \tf{p-\ell}{p} \big) = - \big(
\tf{\ell}{p} \big).$$

Since $p$ is prime, for any $k\in \N$ coprime with $p$, the sets
$\{1,2,\ldots,p-1\}$ and $\{k,2k,\ldots,(p-1)k\}$ coincide modulo $p$
and thus, by $p$-periodicity of the Legendre symbol, we have
$$\sum_{j=1}^{p-1} \big( \tf{k\ell \pm j}{p} \big) =  \sum_{j=1}^{p-1} \big(
\tf{k\ell \pm kj}{p} \big) =
\big( \tf{k}{p} \big) \sum_{j=1}^{p-1} \big( \tf{\ell \pm j}{p} \big) = -
\big( \tf{k}{p} \big) \big( \tf{\ell}{p} \big).$$

On the other hand, if $p\,|\,k$, both sides of \eqref{klj} vanish, and the
first equation in the statement is proved.

\sk For the second equation, splitting the sum $\sum_{j=1}^{2p-1} \big(
\tf{2 \ell \pm j}{p} \big)$ into sums over even and odd
indices, we get
\begin{eqnarray*}
\sum_{j=0}^{p-1} \big( \tf{2 \ell \pm (2j+1)}{p} \big) &=& \sum_{j=1}^{2p-1}
\big( \tf{2 \ell \pm j}{p} \big) -
\sum_{j=1}^{p-1} \big( \tf{2 \ell \pm 2j}{p} \big) = \sum_{j=0}^{p-1} \big(
\tf{2 \ell \pm j}{p} \big) - (\tf 2p) \sum_{j=0}^{p-1}
\big( \tf{\ell \pm j}{p} \big),
\end{eqnarray*}
and by \eqref{klj} we get \eqref{2l-2j+1}.
\end{proof}

\begin{lemma}\label{prop.legendres} Let $p=2q+1$ be a prime and $\ell \in
\N$ with $0\le \ell \le p-1$. Then,
\begin{align*}
(i) & \qquad  \qquad  \begin{array}{lcl}
\sum\limits_{j=1}^{p-1} \big( \tf{\ell + j}{p} \big) j & = &
p \sum\limits_{j=1}^{\ell - 1} \big( \tf{j}{p} \big) +
\sum\limits_{j=1}^{p-1} \big(\tf{j}{p} \big) j,
\msk \\
\sum\limits_{j=1}^{p-1} \big( \tf{\ell - j}{p} \big) j & = & \big(
\tf{-1}{p}\big) \Big( p \sum\limits_{j=1}^{p - \ell - 1}
\big( \tf{j}{p} \big) + \sum\limits_{j=1}^{p-1} \big( \tf{j}{p} \big) j
\Big),
\end{array}
\msk \\
(ii) & \qquad \qquad \begin{array}{lcl}
\sum\limits_{j=1}^{p-1} \big( \tf{2\ell + j}{p} \big) j & = & p
\sum\limits_{j=1}^{2 \ell - \big[\f {2\ell}p \big]\,p - 1} \big( \tf{j}{p}
\big) +
\sum\limits_{j=1}^{p-1} \big(\tf{j}{p} \big) j
\msk \\
\sum\limits_{j=1}^{p-1} \big( \tf{2 \ell - j}{p} \big) j & = & \big(
\tf{-1}{p}\big) \Big( p \sum\limits_{j=1}^{p + \big[\f {2\ell}p \big ]\,p -
2\ell - 1} \big( \tf{j}{p} \big) + \sum\limits_{j=1}^{p-1} \big( \tf{j}{p}
\big) j \Big).
\end{array}
\end{align*}
\end{lemma}

\begin{proof}
If $\ell =0$, then by the class number formula \eqref{eq.dirichlet} there is
nothing to prove. So, we assume that $\ell \ne 0$.
We want to compute the sums $\sum_{j=1}^{p-1} \big( \f{h\ell \pm j}{p} \big)
j$ for $h=1,2$. Let us first consider the case $h=1$. By
Lemma \ref{lem.legendres} we have
\begin{equation}
\label{ljs} \sum_{j=1}^{p-1} \big( \tf{\ell+j}{p} \big) j  =
\sum_{j=1}^{p-1} \big( \tf{\ell+j}{p} \big) (\ell+j) - \ell
\sum_{j=1}^{p-1} \big( \tf{\ell + j}{p} \big)  =  \sum_{j=1}^{p-1} \big(
\tf{\ell+j}{p} \big) (\ell+j) + \ell \big( \tf{\ell}{p} \big).
\end{equation}

Now, since $1\le j \le p-1$, $0\le j \le p-1$, we have $2\le j+\ell \le
2p-2$ and hence $j+\ell$ can be uniquely written as
\begin{equation}\label{j+l}
j+\ell=pq_j + r_j, \qquad 0 \le r_j \le p-1, \: q_j = \left\{
\begin{array}{ll}
0 & \: \text{if } j<p-\ell, \sk \\
1 & \: \text{if } j\ge p-\ell.
\end{array}\right.
\end{equation}
Then, from \eqref{ljs}, and using \eqref{j+l}, we have that
\begin{eqnarray*}
\sum_{j=1}^{p-1} \big( \tf{\ell+j}{p} \big) j &=& \sum_{j=1}^{p-\ell-1}
\big( \tf{r_j}{p} \big) r_j + \sum_{j=p-\ell}^{p-1}  \big(
\tf{r_j}{p} \big) (p+r_j) + \ell \big( \tf{\ell}{p} \big)  \\ &=&
\sum_{j=1}^{p-1} \big( \tf{r_j}{p} \big) r_j + p
\sum_{j=p-\ell}^{p-1} \big( \tf{r_j}{p} \big) + \ell\big( \tf{\ell}{p}
\big).
\end{eqnarray*}
Thus, using that
$$(r_0,r_1,\ldots,r_{p-\ell-1},r_{p-\ell},r_{p-\ell+1},\ldots,r_{p-1}) =
(\ell,\ell+1,\ldots,p-1,0,1,\ldots,\ell-1).$$
we get
\begin{equation}\label{skoll}
\sum_{j=1}^{p-1} \big( \tf{\ell+j}{p} \big) j = \sum_{\begin{smallmatrix}
j=1 \\ j\ne \ell \end{smallmatrix}}^{p-1}
\big( \tf{j}{p} \big) j + p \sum_{j=1}^{\ell-1}  \big( \tf{j}{p} \big) +
\ell \big( \tf{\ell}{p} \big) =  p \sum_{j=1}^{\ell-1}
\big( \tf{j}{p} \big) + \sum_{j=1}^{p-1} \big( \tf{j}{p} \big) j.
\end{equation}
By the previous expression we also have
\begin{equation*}\label{campinas-}
\sum_{j=1}^{p-1} \big( \tf{\ell-j}{p} \big) j =   \big( \tf{-1}{p} \big)
\sum_{j=1}^{p-1} \big( \tf{j+(p-\ell)}{p} \big) j = \big(
\tf{-1}{p} \big) \Big( p \sum_{j=1}^{p-\ell-1}  \big( \tf{j}{p} \big) +
\sum_{j=1}^{p-1} \big( \tf{j}{p} \big) j \Big).
\end{equation*}

Now, consider $h=2$. If $1 \le \ell \le q$ then $2 \le 2\ell \le p-1$ and we
can use \eqref{skoll} directly with $2\ell$ in place of
$\ell$. In the other case, if $q+1 \le \ell \le p-1$ then $1\le 2\ell - p
\le p-2$ and, by \eqref{skoll}, we have
$$\sum_{j=1}^{p-1} \big( \tf{2\ell+j}{p} \big) j =  \sum_{j=1}^{p-1} \big(
\tf{2 \ell - p + j}{p} \big) j
= p \sum_{j=1}^{2\ell-p-1}  \big( \tf{j}{p} \big) + \sum_{j=1}^{p-1} \big(
\tf{j}{p} \big) j.$$ In the remaining case, proceeding as
before and using \eqref{campinas-}, one gets the desired result in the
statement, and thus the proposition follows.
\end{proof}

\begin{lemma}\label{lem.legendres3}
Let $p$ be an odd prime and $\ell \in \N$ with $0\le \ell \le p-1$. Then,
\begin{equation*}\label{legendre odd}
\sum_{j=0}^{p-1} \big( \tf{2\ell \pm (2j+1)}{p} \big) j = \sum_{j=1}^{p-1}
\big( \tf{2\ell \pm j}{p} \big) j - (\tf2p) \sum_{j=1}^{p-1}
\big( \tf{\ell \pm j}{p} \big) j.
\end{equation*}
\end{lemma}
\begin{proof}
We first note that
\begin{eqnarray*}
2\sum_{j=0}^{p-1} \big( \tf{2\ell \mp (2j+1)}{p} \big) j &=&
\sum_{j=0}^{p-1} \big( \tf{2\ell \mp (2j+1)}{p} \big) (2j+1) \\&=&
\sum_{j=1}^{2p-1} \big( \tf{2\ell \mp j}{p} \big) j - \sum_{j=1}^{p-1} \big(
\tf{2\ell \mp 2j}{p} \big) 2j \\ &=& \sum_{j=1}^{p-1}
\big( \tf{2\ell \mp j}{p} \big) j + \sum_{j=p}^{2p-1} \big( \tf{2\ell \mp
j}{p} \big) j - 2 (\tf2p) \sum_{j=1}^{p-1} \big( \tf{\ell \mp
j}{p} \big) j,
\end{eqnarray*}
where in the first equality we have used \eqref{2l-2j+1}.
The second sum in the r.h.s.\@ of the above expression equals
\begin{eqnarray*}
\sum_{h=0}^{p-1} \big( \tf{2\ell \mp (p+h)}{p} \big) (p+h) =  p
\sum_{h=0}^{p-1} \big( \tf{2\ell \mp h}{p} \big) + \sum_{h=0}^{p-1} \big(
\tf{2\ell \mp h}{p} \big) h = \sum_{j=1}^{p-1} \big(
\tf{2\ell \mp j}{p} \big) j.
\end{eqnarray*}
Substituting this expression in the first one we get the desired result.
\end{proof}

We want to compute the sums
\begin{equation}\label{S1yS2}
\begin{split}
& S_1(\ell,p) := \sum_{j=1}^{p-1} \Big( \big( \tf{\ell - j}{p} \big) - \big(
\tf{\ell + j}{p} \big) \Big) j, \\ & S_2(\ell,p) :=
\sum_{j=0}^{p-1} \Big( \big( \tf{2\ell - (2j+1)}{p} \big) - \big( \tf{2\ell
+ (2j+1)}{p} \big) \Big) j,
\end{split}
\end{equation}
for $0\le \ell \le p-1$.
We are now in a position to prove the results that were used in Section 3.
\begin{proposition}\label{prop.dif legendres}
Let $p$ be an odd prime and $\ell \in \N$ with $0\le \ell \le p-1$. Then,
in the  notations in \eqref{aux sums2} we have
\begin{equation*}
S_1(\ell,p) = \left\{ \begin{array}{ll} p \, S_1^-(\ell,p) & \qquad p \equiv
1 \,(4), \sk \\
-p \, S_1^+(\ell,p) - 2\sum\limits_{j=1}^{p-1} (\tf{j}{p})j  & \qquad
p\equiv 3 \,(4), \end{array}\right.
\end{equation*}
and
\begin{equation*}
S_2(\ell,p)  =  \left\{ \begin{array}{lc} p \Big( S_2^-(\ell,p) -  (\tf 2p)
S_1^-(\ell,p) \Big) & p\equiv 1\,(4), \sk \\
-p \Big( S_2^+(\ell,p) -  (\tf 2p) S_1^+(\ell,p) \Big) + 2 \big( (\tf 2p)-1
\big) \sum\limits_{j=1}^{p-1} \big( \tf jp \big)j & p\equiv
3\,(4), \end{array} \right.
\end{equation*}
where $S_h(\ell,p)$ and $S_h^{\pm}(\ell,p)$ are defined in \eqref{S1yS2} and
\eqref{aux sums2} respectively.
\end{proposition}
\begin{proof}
By Lemma \ref{prop.legendres} (i), we have
\begin{eqnarray*}
S_1(\ell,p) & = & \big( \tf{-1}{p}\big) \Big( p \sum_{j=1}^{p - \ell - 1}
\big( \tf{j}{p} \big) + \sum_{j=1}^{p-1} \big( \tf{j}{p}
\big) j \Big) - \Big( p \sum_{j=1}^{\ell - 1} \big( \tf{j}{p} \big) +
\sum_{j=1}^{p-1} \big(\tf{j}{p} \big) j \Big).
\end{eqnarray*}
By using \eqref{legendre2p} and \eqref{aux sums2} we get the first
expression in the statement.

\sk On the other hand, by Lemma \ref{lem.legendres3} we have that {\small
\begin{eqnarray*}
S_2(\ell,p) & = & \Big( \sum_{j=1}^{p-1} \big( \tf{2\ell - j}{p} \big) j -
(\tf2p) \sum_{j=1}^{p-1} \big( \tf{\ell - j}{p} \big) j
\Big) - \Big( \sum_{j=1}^{p-1} \big( \tf{2\ell + j}{p} \big) j - (\tf2p)
\sum_{j=1}^{p-1} \big( \tf{\ell + j}{p} \big) j \Big) \\
&=&  \sum_{j=1}^{p-1} \Big( \big( \tf{2\ell - j}{p} \big)  - \big( \tf{2\ell
+ j}{p} \big) \Big) j -  (\tf2p) \sum_{j=1}^{p-1} \Big(
\big( \tf{\ell - j}{p} \big) - \big( \tf{\ell + j}{p} \big) \Big) j.
\end{eqnarray*}}
By using Lemma \ref{prop.legendres} (ii) we see that $S_2(\ell,p)$ equals
\begin{eqnarray*}
&&(\tf{-1}{p}) \Big( \sum_{j=1}^{p+ \big[\f {2\ell}p \big]\,p-2\ell-1} \big(
\tf{j}{p} \big)  + \sum_{j=1}^{p-1} \big( \tf{j}{p} \big) j \Big)\\&&  -
\Big( \sum_{j=1}^{2\ell-\big [\f {2\ell}p \big ]\,p-1} \big( \tf{j}{p}
\big)  + \sum_{j=1}^{p-1} \big( \tf{j}{p} \big) j \Big) -
(\tf 2p) S_1(\ell,p),
\end{eqnarray*} and now applying \eqref{eq.dirichlet} and using
\eqref{aux sums2} we get the desired result,
and hence the proposition follows.
\end{proof}

\section*{Acknowledgements} The research of P.\@ Gilkey was supported by Project MTM2006-01432 (Spain) and by the
University of C\'ordoba (Argentine). R.\@ Podest\'a wishes to thank the hospitality at the Universidad Aut\'onoma de Madrid (Spain).

\end{document}